\documentclass{amsart}
\usepackage{amssymb}

\usepackage{enumerate}

\usepackage{color}

\newtheorem{theorem}{Theorem}[section]
\newtheorem{lemma}[theorem]{Lemma}

\theoremstyle{definition}
\newtheorem{definition}[theorem]{Definition}

\theoremstyle{remark}
\newtheorem{remark}[theorem]{Remark}

\numberwithin{equation}{section}

\DeclareMathOperator{\divv}{div}

\begin{document}

\title[short text for running head]{Yao XIAO, Dongfen Bian}
\title[Global solution to the nematic liquid crystal flows]{Global solution to the nematic liquid crystal flows with heat effect}


\author{Dongfen Bian}
\address{School of Mathematics and Statistics, Beijing Institute of Technology, Beijing 100081, China; Beijing Key Laboratory on MCAACI, Beijing Institute of Technology, Beijing 100081, China}
\curraddr{}
\email{biandongfen@bit.edu.cn}
\thanks{}

\author{Yao Xiao}
\address{The Institute of Mathematical Sciences, The Chinese University of Hong Kong, Hong Kong}
\email{yxiao@math.cuhk.edu.hk}
\thanks{}

\subjclass[2010]{Primary, 35B35, 35B40, 35B65, 35Q35, 76D03}



\begin{abstract}
The temperature-dependent incompressible nematic liquid crystal flows in a bounded domain  $\Omega\subset\mathbb{R}^N$ ($N=2, 3$) are studied in this paper. Following Danchin's method in [J. Math. Fluid Mech., 2006], we use a localization argument to recover the maximal regularity of Stokes equation with variable viscosity, by which we first prove the local existence of strong solution, then extend it to a global one provided that the initial data is a sufficiently small perturbation around the trivial equilibrium state. This paper also generalizes Hu-Wang's result in [Commun. Math. Phys., 2010] to the non-isothermal case. 
\end{abstract}

\maketitle
\section{Introduction}
Liquid crystal is an intermediate state of matter between isotropic fluids and crystalline solids. Such materials can be artificially obtained typically by increasing the temperature of a solid crystal (low molecular weight) or increasing the concentration of certain solvent (high molecular weight). Among various types of liquid crystals, 
{\textit{nematic}} ($\nu\acute{\eta}\mu\alpha$, thread) ones are those composed of rod-like molecules with head-to-tail symmetry. For more physical and chemical background on the underlying subject, we refer to \cite{Chan, Gennes} and the references therein.

In this paper, we will focus on the mathematical analysis on the following hydrodynamic system, which is a macroscopic continuum description of the time evolution of homogeneous non-isothermal incompressible nematic liquid crystals.
\begin{align}\label{1.1}
\begin{cases}
& \partial_t u+u\cdot\nabla u-\divv\big(\mu(\theta)\mathcal{D}(u)\big)+\nabla P=-\Delta d\cdot\nabla d, \\
& \partial_t d+u\cdot\nabla d=\Delta d+|\nabla d|^2d,  \\
& \partial_t\theta+u\cdot\nabla\theta-\Delta \theta=\frac 12\mu(\theta)|\mathcal{D}(u)|^2+\big|\Delta d+|\nabla d|^2d\big|^2,\\
&\divv u=0,\quad |d|=1.
\end{cases}
\end{align}
The above equations correspond to conservation of linear momentum, angular momentum, internal energy, incompressibility and physical constraint on the director fields. Here, we denote by $u$, $d$, $P$ and $\theta$ the velocity, director, pressure and temperature, respectively. $\mathcal {D}(u)=\frac 12(\nabla u+\nabla^T u)$ is the Cauchy stress  tensor.The total energy density of the system is $e:=\frac 12(|u|^2+|\nabla d|^2)+\theta$.  System \eqref{1.1} is a simplified version of those proposed in \cite{OP, SV}.

Suppose the nematic liquid crystals are confined in a bounded domain $\Omega\subset\mathbb{R}^N (N=2, 3)$, the following initial-boundary conditions are imposed in this paper.
\begin{equation}\label{1.2}
(u, d, \theta)\big|_{t=0}=(u_0, d_0, \theta_0),\quad (u, \mathcal{B}d, \partial_\nu\theta)|_{\partial\Omega}=(0, 0, 0),
\end{equation}
where $\mathcal{B}d=\partial_\nu d$ or $d-{\bf{e}}$, ${\bf{e}}$ is a fixed unit constant vector and $\nu$ is the outward normal vector on $\partial\Omega$. Moreover, $\divv u_0=0$, $|d_0|=1$, $d_0|_{\partial\Omega}={\bf e}$. One easily checks that if $(u, d, P,\theta)$ is a smooth solution to \eqref{1.1}, then under the initial-boundary conditions \eqref{1.2}, the total energy is conserved along with the time evolution:
\begin{equation*}
\frac{d}{dt}\int_\Omega e(t, x)\, dx=0.
\end{equation*}

If neglecting the heat effect and $\mu$ is constant, then \eqref{1.1} reduces to the simplified version of the so-called Ericksen-Leslie system which is developed by Ericksen, Leslie et. al  \cite{Ericksen1, Ericksen-2, Leslie-1, Leslie-2} in the 1960s. For such simplified system, Lin \cite{Lin-1} and Lin-Liu \cite{Lin-2, Lin-3} initiated the study on the Ginzburg-Landau approximated system in 1990s. Specifically, they replaced $|\nabla d|^2d$ by a penalty function $1/{\epsilon^2}(1-|d_\epsilon|^2)d_\epsilon$ and the physical constraint $|d|=1$ is relaxed. For fixed $\epsilon$, they obtained the global well-posedness and partial regularity of the approximated system in two and three dimensional space. 

As for the analysis of the original simplified system, it is more challenging. However, there has been some important results in the two dimensional case, thanks to the local {\it{a priori}} estimates on $\Delta d$ obtained due to a Ladyzhenskja-type inequality by Struwe \cite{Struwe}. In 2010, Lin-Lin-Wang \cite{LLW} proved the existence of global weak solutions, which is regular except for possible finite time slices. At the same time, Hong \cite{Hong} obtained the same results by proving the convergence of the solutions to the approximated system as $\epsilon\to 0_+$. Similar results have been achieved for more general (stress tensor) systems by Hong-Xin \cite{HX}, Huang-Lin-Wang \cite{HLW} and Wang-Wang \cite{WW}. The uniqueness of the above weak solutions is also proved by Lin-Wang \cite{HW}, Xu-Zhang \cite{XZ} and Li-Titi-Xin \cite{LTX}. Also it's worth remarking that  Lei-Li-Zhang \cite{LLZ} generalized Ding-Lin's results on the harmonic maps flow in \cite{DL}, proved that if the initial director satisfies a natural angle condition, then the weak solutions obtained in \cite{ Hong, LLW} are actually smooth. 

For the three dimensional case, the approach by Ladyzhenskja-type inequality fails, little is known for global well-posedness under general large initial data, except that Lin-Wang \cite{LW2} proved the existence of global weak solution if the initial director is targeted on a hemisphere.  As for local well-posedness, there are some results by Hong-Li-Xin \cite{HLX} for the Oseen-Frank model. Also for small perturbations around trivial equilibrium state $(0, {\bf{e}})$, global strong solution is proved in \cite{HNPS}, through a quasilinear approach. And Wang \cite{Wang} obtained the global mild solutions for initial data $(u_0, d_0)$ belonging to possibly the largest space $BMO^{-1}\times BMO$, with small norm.

If heat effect is considered, the system is energetically closed but more complicated. Feireisl-Rocca-Shimperna \cite{F1} and Feireisl-Fremond-Rocca-Schimperna \cite{F2} first studied the approximated system with heat effect and obtained the existence of the global weak solution in two and three dimensional space. Later, Li-Xin \cite{LX} proved the existence of global weak solutions to system \eqref{1.1} in $\mathbb{R}^2$. The uniqueness and regularity of such weak solutions remain open.

As mentioned in the beginning, low-molecular-weight nematic liquid crystals generally are sensitive to the variation of temperature, especially near the threshold of phase transitions.  According to the physical experiment in \cite{GAH}, the principal viscosity of nematic liquid crystals is a continuously differentiable function of temperature in the nematic phase. In particular, we may assume that 
\begin{equation}\label{1.4}
0<\underline{\mu}\le \mu(\theta)\le\bar\mu<\infty,\quad|\mu'(\theta)|\le\bar{\mu}'<\infty\quad\forall~\theta,
\end{equation}
where $\underline\mu$, $\bar\mu$ and $\bar{\mu}'$ are material constants. According to the physical experiments, nematic liquid crystals usually display some instability near the phase transition thresholds, while in the nematic phase, it is generally expected to be stable. In this paper, we give a rigorous mathematical proof for such stability, at least for model \eqref{1.1}. Notice that $(0, \bf{e}, \theta_*)$  is always a trivial equilibrium state to \eqref{1.1}, where $\theta_*$ is arbitrary constant. Without loss of generality, we will set $\theta_*$ to be $0$ from now on. And we will prove that there always exists a unique global strong solution to \eqref{1.1} if the initial data is a sufficiently small perturbation of this trivial equilibrium state.  We would like to remark that recently Hieber-Pr\"uss proved the stability of $(0, {\bf{e}}, 1)$ for a more general model, via a quasilinear approach in \cite{HP}. Also the assumption on the viscosity is different from ours. Here a linear approach is adopted, which the authors also consider to be of independent interest.

Before stating our main result, we set up the functional spaces for strong solutions. Here a {\it strong solution} on $\Omega_T$ means that a set $(u, d, P, \theta)$ satisfying system \eqref{1.1} almost everywhere with initial boundary conditions \eqref{1.2} and belongs to $E^{p, q, r, s}_T$, which is defined as the following.
\begin{definition}\label{def1.1}
For $T>0$ and $1\le p, q, r, s<\infty$, we denote $E^{p, q, r, s}_T$ by the set of $(u, d, P, \theta)$ such that
\begin{align*}
& u\in C([0,T]; D^{1-\frac 1p, p}_{A_r}(\Omega))\cap W^{1,p}(0,T; L^r(\Omega))\cap L^p(0,T;W^{2,r}(\Omega)),\\
& d\in C([0,T]; B^{3-\frac 2p}_{r, p}(\Omega))\cap W^{1,p}(0,T; W^{1,r}(\Omega))\cap L^p(0,T;W^{3,r}(\Omega)),\\
& P\in W^{1,p}(0,T; W^{1,r}(\Omega)),\quad \int_\Omega P\,dx=0,\\
& \theta\in C([0,T]; B^{2-\frac 2s}_{q, s}(\Omega))\cap W^{1,s}(0,T;L^q(\Omega))\cap L^s(0,T;W^{2,q}(\Omega)).
\end{align*}
\end{definition}
Obviously $E^{p, q, r, s}_T$ is a Banach space, we denote it's natural norm as $\|\cdot\|_{E^{p, q, r, s}_T}$. We also remark that the condition 
$$\int_\Omega P\, dx=0$$
in the above definition holds automatically if we replace $P$ by 
$P-\dfrac 1{|\Omega|}\int_\Omega P\, dx $
in system \eqref{1.1}. Also 
$$D^{1-\frac 1p, p}_{A_r}:=(L^r_\sigma, D(A_r))_{1-\frac 1p, p},$$
where
$$L^r_\sigma(\Omega):=\{u\in L^r(\Omega), \divv u=0\},\quad D(A_r)=\{u\in W^{2, r}(\Omega), \divv u=0, u|_{\partial\Omega}=0\}.$$
Moreover, it follows from Proposition 2.5 in \cite{Danchin2006} that
\begin{equation*}
D^{1-\frac 1p, p}_{A_r}\hookrightarrow B^{2(1-\frac 1p)}_{r, p}\cap L^r_\sigma(\Omega).
\end{equation*}
The Besov space $B^{2(1-\frac 1p)}_{r, p}$ on a bounded domain can be regarded as the interpolation space between $L^r$ and $W^{2,r}$, that is, 
\begin{equation*}
B^{2(1-\frac 1p)}_{r, p}=(L^r, W^{2,r})_{1-\frac 1p, p}.
\end{equation*}

We note that this kind of strong solution has been proved to exist for the density-dependent incompressible Navier-Stokes equations by Danchin \cite{Danchin2006} and the simplified Ericksen-Leslie system without the term $|\nabla d|^2d$ by Hu-Wang \cite{HW}. Inspired by their work, we generalize the above results to system \eqref{1.1}. Our first result on the local existence is as follows:

\begin{theorem}\label{local}
Suppose $\Omega\subset\mathbb{R}^N(N=2,3)$ is a bounded domain with smooth boundary. \eqref{1.4} holds true, $u_0\in D^{1-\frac 1p, p}_{A_r}$, $d_0\in B^{3-\frac 2p}_{r, p}$ and $\theta_0\in B^{2-\frac 2s}_{q, s}$ with $p$, $q$, $r$ and $s$ satisfying 
$$1<p<\infty,\quad  2\le s<\infty,\quad N<r\le q,\quad \frac 2p+\frac Nr<\frac 1s+\frac N{2q}+1<2.$$ 
 Then there exists $T_0>0$ such that the system \eqref{1.1}--\eqref{1.2} has a unique local strong solution $(u, d, P, \theta)\in E^{p, q, r, s}_{T_0}$, $T_0$ depends on the initial data.
\end{theorem}
Our second result is on the global existence of strong solutions.
\begin{theorem}\label{global}
Under the conditions of Theorem \ref{local}, in addition, assume that $p\le 2s$, then  there exists $\delta>0$ such that if
\begin{equation*}
\|(u_0,d_0-{\bf e}, \theta_0)\|_{D^{1-\frac 1p, p}_{A_r}\times B^{3-\frac 2 p}_{r, p}\times B^{2-\frac 2 s}_{q, s}}\le \delta,
\end{equation*}
then  system \eqref{1.1}-\eqref{1.2} admits a unique global strong solution $(u, d, P, \theta)\in E^{p, q, r, s}_T$, for any $T>0$ and 
\begin{equation}
\|(u, d-{\bf{e}}, P, \theta)\|_{E^{p , q, r, s}_\infty}\le C\delta,
\end{equation}
for some $C$ independent of initial data and time.
\end{theorem}

\begin{remark}
\begin{enumerate} Concerning the above two theorems, we make following remarks.
\item In the above two theorems, the index set is not empty. In fact, one admissible choice for both is   $s=p=2$ and $q=r>N$. Here we have one more restriction $p\le 2s$ for global existence is due to some time independent interpolation inequality in Lemma \ref{lemma3.2}.

\item It is interesting to compare our result to those in \cite{HLLW}, in which an example of finite time blow-up is given for arbitrarily small $\|u_0\|_{L^2}+\|\nabla d_0\|_{L^2}$. There is no contradiction here since our smallness condition is stronger.

\item Notice that Theorem \ref{global} actually proves the stability of equilibrium state in nematic phase, $(0,{\bf{e}}, 0)$. But whether it is asymptotically stable or not is unknown to the authors at this moment, we hope to address this issue in a future paper. 
\end{enumerate}
\end{remark}
The rest of this paper is organized as follows. In Section 2, we first state the maximal regularity for linear parabolic equation and Stokes equation with variable viscosity. In Section 3,  we establish the local existence and uniqueness of strong solution to system \eqref{1.1} by an iteration method. In Section 4, the global existence of strong solution for small perturbations around the trivial equilibrium state is proved.  Finally, in the appendix, the proof of maximal regularity of Stokes equation with variable viscosity is presented.
\section{The linear estimates}
\subsection{Linear parabolic equation}
First we recall the maximal regularity for the parabolic operators (cf. Theorem 4.10.7 and Remark 4.10.9 in \cite{Amann}):
\begin{theorem}\label{thm2.1}
Given $1<p, q<\infty$, for the Cauchy problem
\begin{align}\label{n2.1}
\begin{cases}
& \partial_t\omega-\Delta \omega=f,\\
& \omega(0)=\omega_0,
\end{cases}
\end{align}
\begin{enumerate}[i)]
\item if $\omega_0\in B^{2-\frac 2p}_{q, p}$ and $f\in L^p(0,T; L^{q}(\mathbb{R}^N))$, then system \eqref{n2.1} has a unique solution $\omega\in W^{1,p}(0,T;L^{q})\cap L^p(0,T; W^{2,q})$ satisfying
\begin{align}\label{n2.2}
\begin{split}
&\|\omega\|_{C([0,T]; B^{2-\frac 2 p}_{q, p})}+\|\omega\|_{W^{1,p}(0,T;L^{q})\cap L^p(0,T;W^{2,q})}\\
&\quad\le C(\|\omega_0\|_{B^{2-\frac 2p}_{q, p}}+\|f\|_{L^p(0,T;L^{q})}),
\end{split}
\end{align}
where $C$ is independent of $\omega_0$, $f$ and $T$.

\item if $\omega_0\in B^{3-\frac 2p}_{q, p}$ and $f\in L^p(0,T; W^{1,q}(\mathbb{R}^N))$, then system \eqref{n2.1} has a unique solution $\omega\in W^{1,p}(0,T;W^{1,q})\cap L^p(0,T; W^{3,q})$ satisfying
\begin{align}\label{n2.3}
\begin{split}
&\|\omega\|_{C([0,T]; B^{3-\frac 2 p}_{q, p})}+\|\omega\|_{W^{1,p}(0,T;W^{1,q})\cap L^p(0,T;W^{3,q})}\\
&\quad\le C(\|\omega_0\|_{B^{3-\frac 2p}_{q, p}}+\|f\|_{L^p(0,T;W^{1,q})}),
\end{split}
\end{align}
where $C$ is independent of $\omega_0$, $f$ and $T$.
\end{enumerate}
\end{theorem}
\begin{remark}
Notice that the above results also hold true for the Neumann or Dirichlet problem on bounded domain with sufficiently regular boundary.
\end{remark}

\subsection{Linearized Stokes equation}
The following theorem plays a key role in our analysis.
\begin{theorem}\label{thm2.4}
Let $\Omega\subset\mathbb{R}^N$ be a bounded domain with smooth boundary, $1<p, r<\infty$, $N<r\le q$ and $\mu$ satisfies \eqref{1.4}. $u_0\in D^{1-\frac 1p, p}_{A_r}$, $f\in L^p(0,T;L^r)$ and $\theta$ satisfies
\begin{equation}\label{2.1}
\theta\in L^\infty(0,T;W^{1,q})\cap \dot{C}^\beta(0,T;L^\infty),
\end{equation}
for some $\beta\in(0,1)$. Then the system
\begin{align}\label{2.2}
\begin{cases}
& \partial_t u-\divv\big(\mu(\theta)\mathcal{D}(u)\big) +\nabla P= f,\\
& \divv u=0, \quad\int_\Omega P\, dx=0,\\
& u|_{t=0}=u_0,\quad u|_{\partial\Omega}=0,
\end{cases}
\end{align}
has a unique solution $(u,P)$ satisfying
\begin{equation}\label{2.3}
\begin{split}
&\|u\|_{C\big([0,t]; D^{1-\frac 1p, p}_{A_r}\big)}+\|(\partial_t u,\Delta u,\nabla P)\|_{L^p_t(L^r)}\\
&\quad\le CB_\theta^{k}(t)e^{CtC_\theta(t)}\big(\|u_0\|_{D^{1-\frac 1p, p}_{A_r}}+\|f\|_{L^p_t(L^r)}\big),
\end{split}
\end{equation}
and
\begin{equation}\label{2.4}
\begin{split}
& \|u\|_{C\big([0,t]; D^{1-\frac 1p, p}_{A_r}\big)}+\|(\partial_t u,\Delta u, u,\nabla P)\|_{L^p(0,t;L^r)}\\&\quad\le C\Big(B_\theta^{k}(t)\big(\|u_0\|_{D^{1-\frac 1p, p}_{A_r}}+\|f\|_{L^p(0,t;L^r)}\big)+C_\theta(t)\|u\|_{L^p(0,t;L^r)}\Big),
\end{split}
\end{equation}
for any $0<t\le T$, where $C$ is independent of $u_0$, $f$, $\theta$ and $T$,
\begin{align*}
& B_\theta(t)=1+\|\nabla\theta\|_{L^\infty_t(L^q(\Omega))}^{\frac q{q-N}},\quad C_\theta(t)=B_\theta^{l_1}(t)(\|\nabla\theta\|_{L^\infty_t(L^q)}+\|\theta\|_{\dot{C}^\beta(0,t;L^\infty)})^{l_2},
\end{align*} 
$k$, $l_1\ge 2$, $l_2\ge 1$ depending only on $p$, $q$, $r$ and $s$.
\end{theorem}

\begin{remark}
As shown in \cite{Danchin2006}, $r>N$ is actually not necessary, we here impose this condition to simplify the index, also this is the case we need in proving our main results. Also as one shall see in the proof, $k\ge 2$ and $k\to\infty$ as $q\to N$, so this estimates does not work for the critical case. Finally notice that $l_2>1$ plays a crucial role in the  proof of global existence of strong solutions.
\end{remark}
The solvability for variable viscosity Stokes equation is well-know in principle, for example we refer to \cite{A, AT, BP, Solo}. Here in order to close the estimates, we need to derive the estimates with explicit dependence on the viscosity. Also for the completeness of our presentation, we give an independent proof in the appendix.
\section{Existence on a small time interval}
Before we proceed, some interpolation inequalities are introduced as preparation. Such inequalities can be also found in \cite{Danchin2006, HW}.
\begin{lemma}\label{lemma3.1}
Under the conditions of Theorem \ref{local}, it holds that
\begin{align}\label{l3.1}
&\|\nabla f\|_{L^p_T(L^\infty)}\le CT^{\frac 12(1-\frac Nr)}\|f\|_{L^\infty_T\big(B^{2-\frac 2 p}_{r, p}\big)}^{\frac p2(1-\frac Nr)}\|f\|_{L^p_T(W^{2,r})}^{1-\frac p2(1-\frac Nr)},\\
& \|\nabla f\|_{L^{2s}_T(L^{2q})}\le CT^{\frac 12(1-\frac 2p+\frac 1s-\frac N r+\frac N{2q})}\|f\|_{L^\infty_T\big(B^{2-\frac 2p}_{r, p}\big)}^{\frac p2(1-\frac Nr+\frac N{2q})}\|f\|_{L^p_T(W^{2,r})}^{1-\frac p2(1-\frac Nr+\frac N{2q})},\label{l3.2}\\
& \|\nabla f\|_{L^{4s}_T(L^{4q})}\le CT^{\frac 12(2-\frac 2p+\frac 1{2s}-\frac Nr+\frac N{4q})}\|f\|_{L^\infty_T\big(B^{3-\frac 2p}_{r, p}\big)}^{\frac p2(2-\frac N r+\frac N{4q})}\|f\|_{L^p_T(W^{3,r})}^{1-\frac p2(2-\frac Nr+\frac N{4q})},\label{l3.21}
\end{align}
where $C$ depends only on $p$, $q$, $r$, $s$ and $\Omega$.
\end{lemma}
\begin{proof}
The proof of the lemma is mainly based on the interpolation and H\"older inequality.
Noticing that  $$B^{2-\frac 2p}_{r, p}\hookrightarrow B^{2-\frac 2p-\frac N r}_{\infty,\infty}, ~~W^{2,r}\hookrightarrow B^{2-\frac N r}_{\infty,\infty},$$ 
$$(B^{1-\frac 2p-\frac Nr}_{\infty,\infty}, B^{1-\frac Nr}_{\infty,\infty})_{\gamma, 1}=B^{s_0}_{\infty, 1}\hookrightarrow L^\infty,$$
where $1-\frac N r-\frac 2p(1-\gamma)=s_0\ge 0$, for some $\gamma\in[0,1)$, then it follows that
\begin{align*}
\|\nabla f\|_{L^p_T(L^\infty)}\le& C\Big(\int_0^T\|\nabla f\|_{B^{s_0}_{\infty, 1}}^p\, dt\Big)^{\frac 1p}\\
\le&C\Big(\int_0^T\|\nabla f\|_{B^{1-\frac 2p-\frac Nr}_{\infty,\infty}}^{p(1-\gamma)}\|\nabla f\|_{B^{1-\frac N r}_{\infty,\infty}}^{p\gamma}\,dt\Big )^{\frac 1p}\\
\le&CT^{\frac {1-\gamma}p}\|f\|_{L^\infty_T(B^{2-\frac 2p}_{r, p})}^{1-\gamma}\|f\|_{L^p_T(W^{2,r})}^\gamma.
\end{align*}
Therefore, \eqref{l3.1} is proved.

 Similarly, notice that $B^{2-\frac 2p}_{r, p}\hookrightarrow B^{2-\frac 2p-\frac Nr+\frac N{2q}}_{2q, p}~(2q>r)$ and 
$W^{2,r}\hookrightarrow B^2_{r, r}\hookrightarrow B^{2-\frac N r+\frac N{2q}}_{2q, r}~(r>N\ge 2)$. On the other hand, 
$$(B^{1-\frac 2p-\frac Nr+\frac N{2q}}_{2q, p},B^{1-\frac N r+\frac N{2q}}_{2q, r})_{\gamma, 1}=B^{s_0}_{2q,1}\hookrightarrow L^{2q},$$
where $1-\frac Nr+\frac N{2q}-\frac 2p(1-\gamma)=s_0\ge 0$, for some $\gamma\in[0,1]$. Consequently, 
\begin{align*}
\|\nabla f\|_{L^{2s}_T(L^{2q})}\le& C\Big(\int_0^T\|\nabla f\|_{B^{s_0}_{2q,1}}^{2s}\,dt\Big)^{\frac 1{2s}}\\
\le & C\Big(\int_0^T\|f\|^{2s(1-\gamma)}_{B^{2-\frac 2p}_{r, p}}\|f\|_{W^{2,r}}^{2s\gamma}\,dt\Big)^{\frac 1{2s}}\\
\le& C\|f\|_{L^\infty_T\big(B^{2-\frac 2 p}_{r, p}\big)}^{1-\gamma}\Big(\int_0^T\|f\|_{W^{2,r}}^{2s\gamma\cdot\frac p{2s\gamma}}\,dt\Big)^{\frac 1{2s}\cdot\frac {2s\gamma}{p}}T^{\frac 1{2s}(1-\frac {2s\gamma}p)}\\
\le& CT^{\frac 1{2s}(1-\frac {2s\gamma}p)}\|f\|_{L^\infty_T\big(B^{2-\frac 2p}_{r, p}\big)}^{1-\gamma}\|f\|_{L^p_T(W^{2,r})}^\gamma.
\end{align*}
In the above computation, H\"older inequality is applicable since $p>2s\gamma$, which is guaranteed by our condition in the theorem.

Finally, notice that $B^{3-\frac 2p}_{r, p}\hookrightarrow B^{3-\frac 2p-\frac Nr+\frac N{4q}}_{4q, p}$, $W^{3,r}\hookrightarrow B^{3-\frac N r+\frac N{4q}}_{4q, r}$ and 
$$(B^{2-\frac 2p-\frac Nr+\frac N{4q}}_{4q, p},B^{2-\frac Nr+\frac N{4q}}_{4q, r})_{\gamma,1}=B^{s_0}_{4q,1}\hookrightarrow L^{4q},$$
where $2-\frac Nr+\frac N{4q}-\frac 2p(1-\gamma)=s_0\ge 0$, for some $\gamma\in [0,1]$. Thus,
\begin{align*}
\|\nabla f\|_{L^{4s}_T(L^{4q})}\le & C\Big(\int_0^T\|\nabla f\|_{B^{s_0}_{4q,1}}^{4s}\,dt\Big)^{\frac 1{4s}}\\
\le& C\Big(\int_0^T\|f\|^{4s(1-\gamma)}_{B^{3-\frac 2p}_{r, p}}\|f\|_{W^{3,r}}^{4s\gamma}\,dt\Big)^{\frac 1{4s}}\\
\le& CT^{\frac 1{4s}(1-\frac{4s\gamma}{p})}\|f\|_{L^\infty_T\big(B^{3-\frac 2p}_{r, p}\big)}^{1-\gamma}\|f\|_{L^p_T(W^{3,r})}^{\gamma},
\end{align*}
where we have used $4s\gamma<p$, in other words, $0\le\frac 2p-\frac 1{2s}+\frac Nr-\frac N{4q}<2$, and this is the direct consequence of our condition.
\end{proof}

\begin{remark}
Since $D^{1-\frac 1p, p}_{A_r}\hookrightarrow B^{2-\frac 2p}_{r, p}$, the space  
$B^{2-\frac 2p}_{r, p}$ in \eqref{l3.1} and \eqref{l3.2} could be replaced  by $D^{1-\frac 1p, p}_{A_r}$ accordingly.
\end{remark}

The above lemma is mainly used to deal with the nonlinear terms of the system in the process of proving the local strong solution for general initial data. However, in the global estimates, some uniform in time control on the nonlinear terms  is needed. To this end, we also have the following version of interpolation inequality:

\begin{lemma}\label{lemma3.2}
Under the conditions of Theorem \ref{global}, it holds that
	\begin{align}
		&\|\nabla f\|_{L^p_T(L^\infty)}\le C\|f\|_{L^p_T(W^{2,r})},\label{l3.3}\\
		&\|\nabla f\|_{L^{2s}_T(L^{2q})}\le C\|f\|_{L^\infty_T\big(B^{2-\frac 2p}_{r, p}\big)}^{1-\frac p{2s}}\|f\|_{L^p_T(W^{2,r})}^{\frac p{2s}},\label{l3.4}\\
		&\|\nabla f\|_{L^{4s}_T(L^{4q})}\le C\|f\|_{L^\infty_T\big(B^{3-\frac 2p}_{r, p}\big)}^{1-\frac p{4s}}\|f\|_{L^p_T(W^{3,r})}^{\frac p{4s}},\label{l3.5}
	\end{align}
	where $C$ depends only on $p$, $q$, $r$, $s$ and $\Omega$.
\end{lemma}
\begin{proof}
\eqref{3.4} immediately follows from the fact that $W^{2, r}\hookrightarrow W^{1,\infty}$ as $r>N$.
Secondly, since $p\le 2s$, by the log-convexity of $L^p$ norms (for example see page 27 in \cite{Adam}),
\begin{align*}
\|\nabla f\|_{L^{2s}_T(L^{2q})}\le \|\nabla f\|_{L^\infty_T(L^{2q})}^{1-\frac p{2s}}\|\nabla f\|_{L^p_T(L^{2q})}^{\frac p{2s}}\le C\|f\|_{L^\infty_T\big(B^{2-\frac 2p}_{r, p}\big)}^{1-\frac p{2s}}\|f\|_{L^p_T(W^{2,r})}^{\frac p{2s}},
\end{align*}
where we have used the fact that $B^{2-\frac 2p}_{r, p}\hookrightarrow W^{1, 2q}$ as $\frac 2p+\frac N r<1+\frac N{2q}$ and $W^{2, r}\hookrightarrow W^{1, 2q}$ as $r>N$, so \eqref{l3.4} is proved. By the same token, one can easily check \eqref{l3.5}.
\end{proof}

Next, we begin to prove the existence of local strong solution through an iteration method. The proof will be divided into the following steps. 

{\bf{Step 1: Construction of approximate solution.}} We initialize the construction of approximate solution by setting $u^0:=u_0$, $d^0:=d_0$ and $\theta^0:=\theta_0$. Given $(u^n, d^n, P^n, \theta^n)\in E^{p, q, r, s}_T$ for some $T>0$, Theorem \ref{thm2.1} enables us to define $\theta^{n+1}$ as the unique solution of the system
\begin{align}\label{3.1}
\begin{cases}
& \partial_t\theta^{n+1}-\Delta\theta^{n+1}\\
&\quad=-u^n\cdot\nabla\theta^n+\frac 12\mu(\theta^n)|\mathcal{D}(u^n)|^2+\big|\Delta \bar d^n+|\nabla \bar d^n|^2(\bar d^n+{\bf{e}})\big|^2,\\
& \theta^{n+1}|_{t=0}=\theta_0,\quad\partial_\nu\theta^{n+1}|_{\partial\Omega}=0,
\end{cases}
\end{align}
on $\Omega_T$, where  $\bar d^{n}:=d^{n}-{\bf{e}}$. Then by Theorem \ref{thm2.1}, define $\bar d^{n+1}$ as the unique solution of system
\begin{align}\label{3.2}
\begin{cases}
\partial_t \bar d^{n+1}-\Delta\bar d^{n+1}=-u^n\cdot\nabla \bar d^n+|\nabla \bar d^n|^2(\bar d^n+{\bf {e}}),\\
\bar d^{n+1}|_{t=0}=d_0-{\bf{e}},\quad \bar d^{n+1}|_{\partial\Omega}=0\quad\text{or}\quad\partial_\nu\bar d^{n+1}|_{\partial\Omega}=0.
\end{cases}
\end{align}
Finally, Theorem \ref{thm2.4} and \eqref{1.4} enables us to define $(u^{n+1}, P^{n+1})$ by $(u^n, d^n,\theta^{n+1})$ as the unique global solution of 
\begin{align}\label{3.3}
\begin{cases}
& \partial_t u^{n+1}-\divv\big(\mu(\theta^{n+1})\mathcal{D} (u^{n+1})\big)+\nabla P^{n+1}=-u^n\cdot\nabla u^n-\Delta \bar d^n\cdot\nabla \bar d^n,\\
& \divv u^{n+1}=0,\quad \int_\Omega P^{n+1}\,dx=0,\\
& u^{n+1}|_{t=0}=u_0,\quad u^{n+1}|_{\partial\Omega}=0.\\
\end{cases}
\end{align}
Also Theorem \ref{thm2.1} and \ref{thm2.4} yield that $(u^{n+1}, \bar d^{n+1}, P^{n+1}, \theta^{n+1})\in E^{p, q, r, s}_T$.

{\bf{Step 2: Uniform bounds for some small fixed time $T$.}} In this step, we aim at finding a positive time $T$ independent of $n$ for which $(u^n, \bar d^n, P^n, \theta^n)_{n\in\mathbb{N}}$ is uniformly bounded in the Banach space $E^{p, q, r, s}_T$.

In order to keep our presentation brief, let us denote
\begin{align*}
& {U}_n(t):= \|u^n\|_{L^\infty_t\big(D^{1-\frac 1p, p}_{A_r}\big)}+\|u^n\|_{W^{1,p}_t(L^r)\cap L^p_t(W^{2,r})}+\|P^n\|_{L^p_t(W^{1,r})},\\
&{D}_n(t):=\|\bar d^n\|_{L^\infty_t\big(B^{3-\frac 2p}_{r, p}\big)}+\|\bar d^n\|_{W^{1,p}_t(W^{1,r})\cap L^p_t(W^{3,r})},\\
& \Theta_n(t):=\|\theta^n\|_{L^\infty_t\big(B^{2-\frac 2 s}_{q, s}\big)}+\|\theta^n\|_{W^{1,s}_t(L^q)\cap L^s_t(W^{2,q})},\\
&{U}_0:=\|u_0\|_{D^{1-\frac 1p, p}_{A_r}},\quad{D}_0:=\|d_0-{\bf{e}}\|_{B^{3-\frac 2p}_{r, p}},\quad\Theta_0:=\|\theta_0\|_{B^{2-\frac 2 s}_{q, s}},\\
&{E}_n(t):={U}_n(t)+{D}_n(t),\quad {E}_0:= {U}_0+{D}_0.
\end{align*}

Then by Theorem \ref{thm2.1}, it holds that
\begin{equation}\label{3.4}
\begin{split}
\Theta_{n+1}(t)\le &C\big(\Theta_0+\underbrace{\|u^n\cdot\nabla\theta^n\|_{L^s_t(L^q)}}_{I_1}+\underbrace{\big\||\nabla u^n|^2\big\|_{L^s_t(L^q)}}_{I_2}\\
&+\underbrace{\big\|\big |\Delta\bar d^n+|\nabla\bar d^n|^2(\bar d^n+{\bf{e}})\big|^2\big\|_{L^s_t(L^q)}}_{I_3}\big),
\end{split}
\end{equation}
for any $t>0$.
Next, we evaluate the terms on the RHS of \eqref{3.4} one by one. It is noted that in the following the constant $\gamma\in [0,1]$ may vary in different inequalities and its value does not play a role in our analysis, thus from now on we do not distinguish them in notation unless otherwise claimed.

The first term on the RHS of \eqref{3.4} can be estimated as
\begin{align}{\label{3.5}}
\begin{split} 
I_1\le& C\|u^n\|_{L^\infty_t(L^q)}\|\nabla\theta^n\|_{L^s_t(L^\infty)}\\
\le& Ct^{\frac 12(1-\frac Nq)}\|u^n\|_{L^\infty_t\big(D^{1-\frac 1p, p}_{A_r}\big)}\|\theta^n\|_{L^\infty_t(B^{2-\frac 2s}_{q, s})}^{1-\gamma}\|\theta^n\|_{L^s_t(W^{2,q})}^\gamma\\
\le& Ct^{\frac 12(1-\frac Nq)}{U}_n(t)\Theta_n(t),
\end{split}
\end{align}
where we have used the fact $D^{1-\frac 1p, p}_{A_r}\hookrightarrow L^q$ as $\frac 2p+\frac N r<2+\frac N q$ and inequality \eqref{l3.1}.

For $I_2$, it follows from \eqref{l3.2} that
\begin{align}{\label{3.6}}
\begin{split}
I_2\le& C\|\nabla u^n\|_{L^{2s}_t(L^{2q})}^2\\
\le& C t^{1-\frac 2p+\frac 1s-\frac N r+\frac N{2q}}\|u^n\|_{L^\infty_t\big(D^{1-\frac 1p, p}_{A_r}\big)}^{2(1-\gamma)}\|u^n\|_{L^p_t(W^{2,r})}^{2\gamma}\\
\le& C t^{1-\frac 2p+\frac 1s-\frac N r+\frac N{2q}}{U}_n^2(t).
\end{split}
\end{align}
Next, we evaluate the last term as 
\begin{align*}
\begin{split}
I_3\le& C\big(\|\Delta\bar d^n \|_{L^{2s}_t(L^{2q})}^2+\big\| |\nabla \bar d^n|^2(\bar d^n+{\bf{e}})\big\|_{L^{2s}_t(L^{2q})}^2\big)\\
\le& C\big(\|\Delta \bar d^n\|_{L^{2s}_t(L^{2q})}^2+\|\nabla\bar d^n\|_{L^{4s}_t(L^{4q})}^4(\|\bar d^n\|_{L^\infty_t(L^\infty)}^2+1)\big).
\end{split}
\end{align*}
Since
\begin{align*}
\begin{split}
\|\Delta \bar d^n\|_{L^{2s}_t(L^{2q})}\le& C t^{\frac 12(1-\frac 2p+\frac 1s-\frac N r+\frac N{2q})}\|\nabla \bar d^n\|^{1-\gamma}_{L^\infty_t\big(B^{2-\frac 2p}_{r, p}\big)}\|\nabla\bar d^n\|_{L^p_t(W^{2,r})}^\gamma\\
\le & Ct^{\frac 12(1-\frac 2p+\frac 1s-\frac N r+\frac N{2q})}{D}_n(t),\\
\|\nabla\bar d^n\|_{L^{4s}_t(L^{4q})}\le&C t^{\frac 12(2-\frac 2p+\frac 1{2s}-\frac Nr+\frac N{4q})}\|\bar d^n\|_{L^\infty_t\big(B^{3-\frac 2p}_{r, p}\big)}^{1-\gamma}\|\bar d^n\|_{L^p_t(W^{3,r})}^{\gamma}\\
\le& C t^{\frac 12(2-\frac 2p+\frac 1{2s}-\frac Nr+\frac N{4q})}{D}_n(t),\\
\|\bar d^n\|_{L^\infty(L^\infty)}\le& C\|\bar d^n\|_{L^\infty\big(B^{3-\frac 2p}_{r, p}\big)}\le C{D}_n(t),
\end{split}
\end{align*}
where we have used inequality \eqref{l3.21} and $B^{3-\frac 2p}_{r, p}\hookrightarrow L^\infty$ as $\frac 2p+\frac Nr<3$. 
Thus, 
\begin{equation}\label{3.8}
I_3\le Ct^{1-\frac 2p+\frac 1s-\frac N r+\frac N{2q}}{D}_n^2(t)+C t^{2(2-\frac 2p+\frac 1{2s}-\frac Nr+\frac N{4q})}{D}_n^4(t)(1+{D}_n^2(t)).
\end{equation}

Substituting \eqref{3.5}, \eqref{3.6} and \eqref{3.8}  into \eqref{3.4} yields that
\begin{equation}\label{3.12}
\Theta_{n+1}(t)\le C\big(\Theta_0+t^{\xi_1}(\Theta^2_n(t)+{E}_n^2(t)+{E}_n^6(t))\big),
\end{equation}
where $\xi_1=\min\{\frac 12(1-\frac Nq),1-\frac 2p+\frac 1s-\frac Nr+\frac N{2q},2(2-\frac 2p+\frac 1{2s}-\frac Nr+\frac N{4q})\}>0$.

Next, we move on to evaluate ${D}_{n+1}(t)$. Applying Theorem \ref{thm2.1} to \eqref{3.2}, one obtains
\begin{equation}\label{3.13}
{D}_{n+1}(t)\le C\big({D}_0+\underbrace{\|u^n\cdot\nabla \bar d^n\|_{L^p_t(W^{1,r})}}_{{II}_1}+\underbrace{\big\| |\nabla \bar d^n|^2(\bar d^n+{\bf{e}})\big\|_{L^p_t(W^{1,r})}}_{II_2}\big).
\end{equation}
For the first term on the RHS of \eqref{3.13}, 
\begin{align}\label{3.17}
II_1\le& \|u^n\cdot\nabla \bar d^n\|_{L^p_t(L^r)}+\|\nabla u^n\cdot\nabla \bar d^n\|_{L^p_t(L^r)}+\|u^n\cdot\nabla^2 \bar d^n\|_{L^p_t(L^r)}\nonumber\\
\le & C(\|u^n\|_{L^\infty_t(L^r)}\|\nabla\bar d^n\|_{L^p_t(W^{1,\infty})}+\|\nabla u^n\|_{L^p_t(L^\infty)}\|\nabla\bar d^n\|_{L^\infty_t(L^r)})\\
\le & C t^{\frac 12(1-\frac Nr)}U_n(t)D_n(t).\nonumber
\end{align}
For the last term on the RHS of \eqref{3.13}, by \eqref{l3.3}, one can obtain
\begin{align}{\label{3.18}}
II_2\le &\big\| |\nabla \bar d^n|^2(\bar d^n+{\bf{e}})\big\|_{L^p_t(L^r)}+2\|\nabla^2 \bar d^n\nabla \bar d^n (\bar d^n+{\bf{e}})\|_{L^p_t(L^r)}+\big\| |\nabla \bar d^n|^3\big\|_{L^p_t(L^r)}\nonumber\\
\le& C(1+\|\bar d^n\|_{L^\infty_t(W^{1,\infty})})\|\nabla\bar d^n\|_{L^\infty_t(L^r)}\|\nabla\bar d^n\|_{L^p_t(W^{1,\infty})}\\
\le & Ct^{\frac 12(1-\frac Nr)} {D}_n^2(t)(1+{D}_n(t))\nonumber.
\end{align}
Therefore, substituting \eqref{3.17} and \eqref{3.18} into \eqref{3.13} gives that
\begin{equation}\label{3.19}
{D}_{n+1}(t)\le C\big({D}_0+t^{\xi_2}({E}_n^2(t)+{E}_n^3(t))\big),
\end{equation}
for $\xi_2=\frac 12(1-\frac Nr)$.

Finally, applying Theorem \ref{thm2.4} to \eqref{3.3}, we get
\begin{align}{\label{3.26}}
{U}_{n+1}(t)\le & C (1+\|\nabla\theta^{n+1}\|_{L^\infty_t(L^{q})})^{\varsigma}\exp\big(Ct(1+\|\theta^{n+1}\|_{\dot{C}^\beta_t(L^\infty)}+\|\nabla\theta\|_{L^\infty_t(L^q)})^\varsigma\big)\nonumber\\
&\times({U}_0+\|u^n\cdot\nabla u^n\|_{L^p_t(L^r)}+\|\Delta \bar d^n\cdot\nabla\bar d^n\|_{L^p_t(L^r)}),
\end{align}
for some $\beta\in (0,1)$ and $1<\varsigma<\infty$ depending on $p$, $q$, $r$ and $s$.

Since $s\ge2$ and $q>N$, by Sobolev embedding,
\begin{equation*}
W^{1,s}_t(L^q)\cap L^s(W^{2,q})\hookrightarrow L^\infty_t(W^{1,q})\cap \dot{C}^\beta_t(L^\infty),\quad\beta=1-\frac 1s-\frac N{2q}\in(0,1).
\end{equation*}
Thus, \eqref{3.26} reduces to
\begin{align}\label{3.27}
\begin{split}
{U}_{n+1}(t)\le & C (1+\Theta_{n+1}(t))^\varsigma \exp \big(Ct(1+\Theta_{n+1}(t))^\varsigma\big)\\
&\times\big({U}_0+\underbrace{\|u^n\cdot\nabla u^n\|_{L^p_t(L^r)}}_{III_1}+\underbrace{\|\Delta \bar d^n\cdot\nabla \bar d^n\|_{L^p_t(L^r)}}_{III_2}\big).
\end{split}
\end{align}
By the interpolation inequality \eqref{l3.1}, it follows that
\begin{align}
III_1\le & C\|u^n\|_{L^\infty_t(L^r)}\|\nabla u^n\|_{L^p_t(L^\infty)}\nonumber\\
\le & C t^{\frac 12(1-\frac Nr)}\|u^n\|_{L^\infty_t\big(D^{1-\frac 1p, p}_{A_r}\big)}^{2-\gamma}\|u^n\|_{L^p_t(W^{2,r})}^\gamma\nonumber\\
\le & C t^{\frac 12(1-\frac Nr)}{U}_n^2(t),\label{3.28}\\
III_2\le & C\|\Delta\bar d^n\|_{L^p_t(L^\infty)}\|\nabla\bar d^n\|_{L^\infty_t(L^r)}\nonumber\\
\le & Ct^{\frac 12(1-\frac Nr)}\|\bar d^n\|_{L^\infty_t\big(B^{3-\frac 2p}_{r, p}\big)}^{2-\gamma}\|\bar d^n\|_{L^p_t(W^{3,r})}^\gamma\nonumber\\
\le & Ct^{\frac 12(1-\frac Nr)}{D}_n^2(t).\label{3.29}
\end{align}
Substituting \eqref{3.28} and \eqref{3.29} into \eqref{3.27}, one reaches
\begin{align}\label{3.30}
{U}_{n+1}(t)\le& C(1+\Theta_{n+1}(t))^\varsigma\exp\big(Ct(1+\Theta_{n+1}(t))^\varsigma\big)\big({U}_0+t^{\xi_2}{E}_n^2(t)\big).
\end{align}

Adding up \eqref{3.19} and \eqref{3.30}, one infers that
\begin{equation}\label{3.31}
{E}_{n+1}(t)\le C(1+\Theta_{n+1}(t))^\varsigma \exp\big(Ct(1+\Theta_{n+1}(t))^\varsigma\big)\big({E}_0+t^{\xi_2} ({E}_n^3(t)+E_n^2(t))\big).
\end{equation}
Assume that for some $T>0$ such that for any $t\in[0, T]$,
\begin{equation*}
\Theta_n(t)\le CM_1(\Theta_0+E_0),\quad E_n(t)\le CM_2(\Theta_0+E_0),
\end{equation*}
where $M_1$ and $M_2$ are some constants independent of $T$ and to be determined later.

Choosing $0<T_1\le T$ such that 
\begin{equation}
T_1^{\xi_1}\big(M_1^2C^2(\Theta_0+E_0)+M_2^2C^2(\Theta_0+E_0)+M_2^6C^6(\Theta_0+E_0)^5\big)\le 1,
\end{equation}
then for any $t\in [0, T_1]$, it follows from \eqref{3.12} that
\begin{equation}
\Theta_{n+1}(t)\le 2C(\Theta_0+E_0).
\end{equation}
Choosing $0<T_2\le T_1$ such that
\begin{align}
\begin{cases}
& CT_2(1+2C(\Theta_0+E_0))^\varsigma\le\ln 2,\\
&T_2^{\xi_2}\big(M_2^2C^2(\Theta_0+E_0)+M_2^3C^3(\Theta_0+E_0)^2\big)\le 1,
\end{cases}
\end{align}
then for any $t\in[0, T_2]$,
\begin{equation}
E_{n+1}(t)\le 4C\big(1+2C(\Theta_0+E_0)\big)^\varsigma(\Theta_0+E_0).
\end{equation}
Now choosing 
\begin{equation*}
M_1=2,\quad M_2=4\big(1+2C(\Theta_0+E_0)\big)^\varsigma,
\end{equation*}
thus one has
\begin{equation*}
\Theta_{n+1}(t)\le CM_1(\Theta_0+E_0),\quad E_{n+1}(t)\le CM_2(\Theta_0+E_0),\quad\forall~t\in[0, T_2]
\end{equation*}
Then by the induction argument, $(u^n, d^n, P^n, \theta^n)$   is uniformly bounded in $E^{p, q, r, s}_{T_2}$ with respect to $n$.

{\bf{Step 3: Convergence of sequence in $E^{p, q, r, s}_{T}$ for some $T<T_2$.}} In this step, we are devoted to proving that $(u^n, d^n, P^n, \theta^n)$ is a Cauchy sequence in the Banach space $E^{p, q, r, s}_T$ for sufficiently small $T<T_2$.

Let $\delta u^n:=u^{n+1}-u^n$, $\delta P^n:=P^{n+1}-P^n$, $\delta d^n:=d^{n+1}-d^n$, $\delta \theta^n:=\theta^{n+1}-\theta^n$, $\delta\mu^n:=\mu(\theta^{n+1})-\mu(\theta^n)$, $\delta\Theta_n:=\Theta_{n+1}-\Theta_n$ and $\delta{E}_n:={E}_{n+1}-{E}_n$. Then $(\delta u^n, \delta d^n, \delta P^n, \delta \theta^n)$ satisfies the system
\begin{align}\label{3.35}
\begin{cases}
& \partial_t\delta u^n-\divv\big(\mu(\theta^{n+1})\mathcal{D}(\delta u^n)\big)+\nabla \delta P^n\\
&\quad=\divv(\delta\mu^n\mathcal{D}(u^n))-\delta u^{n-1}\cdot\nabla u^n-u^{n-1}\cdot\nabla\delta u^{n-1}\\
&\quad\quad-\Delta d^n\cdot\nabla \delta d^{n-1}-\Delta \delta d^{n-1}\cdot\nabla d^{n-1},\\
&\partial_t\delta d^n-\Delta \delta d^n=|\nabla d^n|^2\delta d^{n-1}+\nabla d^n\nabla \delta d^{n-1} d^{n-1}\\
&\quad +\nabla \delta d^{n-1}\nabla d^{n-1} d^{n-1}-\delta u^{n-1}\cdot\nabla d^n-u^{n-1}\cdot\nabla\delta d^{n-1},\\
&\partial_t\delta \theta^n-\Delta \delta\theta^n=-u^n\cdot\nabla\delta\theta^{n-1}-\delta u^{n-1}\cdot\nabla\theta^{n-1}\\
&\quad+\frac 12\delta\mu^{n-1}|\mathcal{D}(u^n)|^2+\frac 12\mu^{n-1}\mathcal{D}(\delta u^{n-1}):(\mathcal{D}(u^n)+\mathcal{D}(u^{n-1}))\\
&\quad+\big(\Delta d^n+|\nabla d^n|^2d^n+\Delta d^{n-1}+|\nabla d^{n-1}|^2d^{n-1}\big)\\
&\quad\times\big(\Delta\delta d^{n-1}+|\nabla d^n|^2\delta d^{n-1}+(\nabla d^n+\nabla d^{n-1})\nabla\delta d^{n-1}d^{n-1}\big)\\
& \divv \delta u^n=0,\quad \int_\Omega \delta P^n\, dx=0,\\
& (\delta u^n, \delta d^n, \delta \theta^n)|_{t=0}=(0,0,0),\quad (\delta u^n,\mathcal{B}\delta d^n, \partial_\nu\delta\theta^n)|_{\partial\Omega}=(0,0,0).
\end{cases}
\end{align}

Applying Theorem \ref{thm2.1} and Theorem \ref{thm2.4} to system \eqref{3.35}, mimicking the process in the second step and noticing that $\Theta_n(t)+{E}_{n}(t)\le C(\Theta_0+E_0)$, for any $t<T_2$ and $n\in\mathbb{N}$, one obtains
\begin{align}
\begin{cases}
& \delta\Theta_n(t)\le Ct^{\xi_1}(\delta\Theta_{n-1}(t)+\delta E_{n-1}(t)),\\
& \delta E_n(t)\le Ct^{\xi_2}(\delta\Theta_{n-1}(t)+\delta E_{n-1}(t)),
\end{cases}
\end{align}
for any $t<T_2$, $C$ depends on the domain, $p$, $q$, $r$, $\underline\mu$, $\bar\mu$, $\overline{\mu}'$, $\Theta_0$ and ${E}_0$.
Choosing $T_3\in(0, T_2]$ such that 
\begin{equation}\label{time2}
C(T_3^{\xi_1}+T_3^{\xi_2})\le\frac 12,
\end{equation}
then for any $t<T_3$, it holds that
\begin{equation*}
\delta\Theta_n(t)+\delta{E}_n(t)\le \frac 12(\delta\Theta_{n-1}(t)+\delta{E}^{n-1}(t)).
\end{equation*}
Therefore, $(\delta u^n,\delta d^n,\delta P^n,\delta \theta^n)$ is a Cauchy sequence in the Banach space $E_{T}^{p, q, r, s}$ for any $T\le T_3$. 

{\bf{Step 4: Verifying that the limit is a local strong solution.}} Let $(u, d, P, \theta)\in E^{p, q, r, s}_T$ be the limit of $(u^n, d^n, P^n, \theta^n)_{n\in\mathbb{N}}$  for $T\le T_3$.

We claim that all the nonlinear terms of \eqref{3.1} converge to their corresponding terms of \eqref{1.1} in $L^s_{T_3}(L^q)$. We take one term as an example,
\begin{align*}
 &\big\|\mu(\theta^n)|\nabla u^n|^2-\mu(\theta)|\nabla u|^2\big\|_{L^s_{T_3}(L^q)}\\
 &\quad\le 2\bar\mu\big\|\nabla u^n-\nabla u\big\|_{L^{2s}_{T_3}(L^{2q})}\Big(\|\nabla u\|_{L^{2s}_{T_3}(L^{2q})}+\|\nabla u^n\|_{L^{2s}_{T_3}(L^{2q})}\Big)\\
 &\quad\quad+\bar{\mu'}\big\|(\theta^n-\theta)|\nabla u^n|^2\big\|_{L^s_{T_3}(L^q)}\\
 &\quad\le  C\big({E}_n(T_3)+{E}(T_3)\big)\big(\|u^n-u\|_{E^{p, q, r, s}_{T_3}}+\|\theta^n-\theta\|_{E^{p, q, r, s}_{T_3}}\big)\\
 &\quad \longrightarrow 0,\quad\text{as}\quad n\to \infty.
\end{align*}
The convergence of the rest terms can be proved in the same spirit as above. Similarly, the nonlinear terms of \eqref{3.2} converges in the space $L^p_{T_3}(W^{1,r})$ and ones of \eqref{3.3} converges in the space  $L^p_{T_3}(L^r)$. Therefore, we can perform the limiting process and it's easy to verify that the limits indeed satisfy system \eqref{1.1} with \eqref{1.2} almost everywhere.

Next, we check that $|d|=1$ {\it a.e.} on $\Omega_{T_3}$. In particular, consider the equations
\begin{align}\label{3.43}
\begin{cases}
&\partial_ t d-\Delta d+u\cdot\nabla d=|\nabla d|^2d,
\\&d|_{t=0}=d_0,\quad\text{on}\quad \Omega,\\
&d=d_0,\quad\text{or}\quad\partial_\nu d=0,\quad\text{on}\quad \partial\Omega\times[0,T_3). 
\end{cases}
\end{align} 
Since $d\in L^\infty_{T_3}\big(B^{3-\frac 2p}_{r, p}\big)\cap L^p_{T_3}(W^{3,r})$, multiplying \eqref{3.43} by $d$ and using the fact that $\Delta(|d|^2)=2\Delta d\cdot d+2|\nabla d|^2$, one obtains
\begin{align}\label{3.44}
\begin{cases}
&\partial_t (|d|^2-1)-\Delta (|d|^2-1)+u\cdot\nabla (|d|^2-1)=2|\nabla d|^2(|d|^2-1),\\
& (|d|^2-1)|_{t=0}=0,\quad\text{on}\quad  \Omega,\\
&|d|^2-1=0,\quad\text{or}\quad\partial_\nu(|d|^2-1)=0, \quad\text{on}\quad \partial\Omega\times[0,T_3).
\end{cases}
\end{align}
Multiplying \eqref{3.44} by $|d|^2-1$, then integrating the resulting equation over $\Omega$, one has
\begin{equation}
\frac12\dfrac d{dt}\int_\Omega\big(|d|^2-1\big)^2dx+\int_{\Omega}|\nabla(|d|^2-1)|^2dx=2\int_\Omega |\nabla d|^2(|d|^2-1)^2dx.
\end{equation}  
Notice that if $p\ge2$, $L^p_{T_3}(W^{3,r})\hookrightarrow L^2_{T_3}(W^{1,\infty})$, and if $1<p<2$, $W^{1,p}_{T_3}(W^{1,r})\cap L^p_{T_3}(W^{3,r})\hookrightarrow L^2_{T_3}(W^{1,\infty})$. Thus, by Gronwall's inequality, one has
\begin{equation}
\int_\Omega(|d|^2-1)^2(t, x)dx\le \exp \big(C\int_0^t\|\nabla d(s,\cdot)\|_{L^\infty}^2ds\big)\int_\Omega(|d|^2-1)^2(0,x)dx=0,
\end{equation} 
 for any $t<T_3$. This implies that $|d|=1$ {\it a.e.} on $\Omega_{T_3}$.

Finally, the existence of local strong solution is proved. The proof of the uniqueness  and continuity is standard, here we omit the details.

\section{Global existence for small perturbation}
 It is showed in the last section in implicit form that if the initial data $\Theta_0$ and  $E_0$ are smaller, the lifespan of local strong solutions is longer. In this section we are going to prove that actually for sufficiently small perturbation around the trivial equilibrium state $(0,{\bf{e}}, 0)$, the strong solution is global.

\begin{proof}[Proof of Theorem \ref{global}]
Suppose $T^*$ is the maximal existence time and fix $t<T^*$.
Define
\begin{align*}
&{U}(t):=\|u\|_{L^\infty_t\big(D^{1-\frac 1p, p}_{A_r}\big)}+\|u\|_{W^{1,p}_t(L^r)\cap L^p_t(W^{2,r})}+\|P\|_{L^p_t(W^{1,r})},\\
&{D}(t):=\|\bar d\|_{L^\infty_t\big(B^{3-\frac 2p}_{r, p}\big)}+\|\bar d\|_{W^{1,p}_t(W^{1,r})\cap L^p_t(W^{3,r})},\\
&\Theta(t):=\|\theta\|_{L^\infty_t\big(B^{2-\frac 2s}_{q, s}\big)}+\|\theta\|_{W^{1,s}_t(L^q)\cap L^s_t(W^{2,q})},\\
&F(t):={U}(t)+{D}(t)+\Theta(t),\quad F_0:=U_0+D_0+\Theta_0,
\end{align*}
where ${U}_0$, $D_0$ and $\Theta_0$ are defined as before, $\bar d:=d-{\bf{e}}$.

Applying Theorem \ref{thm2.1} to the temperature equation in system \eqref{1.1}, one has
\begin{align}\label{4.1}
\begin{split}
\Theta(t)\le& C\big(\Theta_0+\|u\cdot\nabla\theta\|_{L^s_t(L^q)}+\bar\mu\big\| |\nabla u|^2\big\|_{L^s_t(L^q)}\\
&+\big\| |\Delta\bar d+|\nabla\bar d|^2(\bar d+{\bf{e}})|^2\big\|_{L^s_t(L^q)}\big)\\
\le& C\big(\Theta_0+{U}(t)\Theta(t)+{U}^2(t)+{D}^2(t)+{D}^6(t)\big),
\end{split}
\end{align}
where we have used Lemma \ref{lemma3.2}. Then applying Theorem \ref{thm2.1} to the second equation of \eqref{1.1}, it follows that
\begin{align}\label{4.2}
{D}(t)\le& C\big({D}_0+\|u\cdot\nabla\bar d\|_{L^p_t(W^{1,r})}+\big\| |\nabla\bar d|^2(\bar d+{\bf e}) \big\|_{L^p_t(W^{1,r})}\big)\nonumber\\
\le& C({D}_0+\|u\|_{L^\infty_t(L^r)}\|\nabla\bar d\|_{L^p_t(W^{1,\infty})}+\|\nabla u\|_{L^p_t(L^\infty)}\|\nabla\bar d\|_{L^\infty_t(L^r)}\nonumber\\
&+\|\nabla\bar d\|_{L^p_t(W^{1,\infty})}\|\nabla\bar d\|_{L^\infty_t(L^r)}(1+\|\bar d\|_{L^\infty_t(W^{1,\infty})}))\nonumber\\
\le& C({D}_0+U(t)D(t)+D^2(t)+D^3(t)),
\end{align}
where we have used Lemma \ref{lemma3.2} and $B^{3-\frac 2p}_{r, p}\hookrightarrow W^{1,\infty}$. Applying Theorem \ref{thm2.4} to the velocity equation in system \eqref{1.1} and using \eqref{2.4}, one reaches
\begin{align}
{U}(t)\le& C(1+\Theta(t))^{\frac{k q}{q-N}}({U}_0+\|u\cdot\nabla u\|_{L^p_t(L^r)}+\|\Delta\bar d\cdot\nabla\bar d\|_{L^p_t(L^r)})\nonumber\\
&\quad+(1+\Theta(t))^{\frac{q l_1}{q-N}}(\|\nabla\theta\|_{L^\infty_t(L^q)}+\|\theta\|_{\dot{C}^\beta_t(L^\infty)})^{l_2}\|u\|_{L^p_t(L^r)}.
\end{align}
Noticing that $W^{1, s}_t(L^q)\cap L^s_t(W^{2, q})\hookrightarrow \dot{C}^\beta_t(L^\infty)$, it follows that 
\begin{equation}\label{4.4}
U(t)\le C(1+\Theta(t))^{\max\{k, l_1\}\frac q{q-N}}(U_0+U^2(t)+D^2(t)+\Theta^{l_2}(t)U(t)).
\end{equation}
Summing up \eqref{4.1}, \eqref{4.2} and \eqref{4.4}, it yields that 
\begin{equation}\label{4.5}
F(t)\le C(1+F(t))^{\max\{k, l_1\}\frac q{q-N}}(F_0+F^2(t)+F^3(t)+F^6(t)+F^{1+l_2}(t)).
\end{equation}
Assume that for some fixed $T>0$ such that
\begin{equation}\label{4.6}
F(t)\le 5CF_0,\quad\forall~t\in[0,T].
\end{equation}
If the initial data is sufficiently small such that 
\begin{align}\label{small}
\begin{cases}
& (1+5CF_0)^{\max\{k, l_1\}\frac{q}{q-N}}\le 2,\\
& 5^2C^2F_0+5^3C^3F_0^2+M^6C^6F_0^5+5^{1+l_2}C^{1+l_2}F_0^{l_2}\le 1,
\end{cases}
\end{align}
then one can be convinced that, by \eqref{4.5}, 
\begin{equation}
F(t)\le 4CF_0,\quad \forall~t\in[0, T].
\end{equation}

Thus by a continuation argument, one can extend a local solution to a global one. $\delta$ in the smallness condition of initial data is determined by \eqref{small}.
\end{proof}
\section{Appendix: Proof of Theorem \ref{thm2.4}}
 First of all, recall the maximal regularity for the linear Stokes operator (cf. Theorem 3.2 \cite{Danchin2006}):
\begin{theorem}\label{thm2.3}
Let $\Omega\subset\mathbb{R}^N$ be a bounded domain with smooth boundary  and $1<p, r<\infty$. $u_0\in D^{1-\frac 1p, p}_{A_r}$, $f\in L^p(0,T; L^r)$ and $\mu$ is a positive constant. Then the system 
\begin{align*}
\begin{cases}
\partial_t u-\mu\Delta u+\nabla P=f,\\
\divv u=0,\quad \int_\Omega P\,dx=0,\\
u|_{t=0}=u_0,\quad u|_{\partial\Omega}=0,
\end{cases}
\end{align*}
has a unique global solution $(u, P)$ satisfying 
\begin{align*}
& \mu^{1-\frac 1p} \|u\|_{L^\infty(0,T;D^{1-\frac 1p,p}_{A_r})}+\|(\partial_t u,\mu\Delta u, u,\nabla P)\|_{L^p(0,T;L^r)}\\
 &\quad\le C\Big(\mu^{1-\frac 1p}\|u_0\|_{D^{1-\frac 1p,p}_{A_r}}+\|f\|_{L^p(0,T;L^r)}\Big),
\end{align*}
for all $T\ge0$, with $C=C(p,r,\sigma(\Omega))$, where $\sigma(\Omega)$ stands for the open set
\begin{equation*}
\sigma(\Omega)=\Big\{\dfrac x{\delta(\Omega)}\Big|x\in\Omega\Big\},
\end{equation*}
with $\delta(\Omega)$ denoting the diameter of $\Omega$.
\end{theorem}
\begin{remark}
Notice that in the above estimates $C$ depends on the shape of the domain $\Omega$, but is independent of the diameter $\delta(\Omega)$.
\end{remark}

The basic idea is that if $\theta$ is close to a constant $\underline\theta$, Theorem \ref{thm2.3} provides us with the desired estimates. Indeed, one can rewrite the system as 
\begin{align}\label{n2.5}
\begin{cases}
& \partial_t u-\divv\big(\mu(\underline\theta)\mathcal{D}(u)\big) +\nabla P= f+\big(\mu(\theta)-\mu(\underline\theta)\big)\Delta u+\mu'(\theta)\nabla\theta\cdot\mathcal{D}(u),\\
& \divv u=0, \quad\int_\Omega P\, dx=0,\\
& u|_{t=0}=u_0,\quad u|_{\partial\Omega}=0.
\end{cases}
\end{align}
Now if $\|\mu(\theta)-\mu(\underline\theta)\|_{L^\infty}$ is small, the term $\|\big(\mu(\theta)-\mu(\underline\theta)\big)\Delta u\|_{L^p_t(L^r)}$  may be absorbed by the left-hand side (LHS) of the inequality given in Theorem \ref{thm2.3}.  Although generally one can not expect $\theta$ is close enough to a constant, but if $\theta$ is Lipschitz continuous, it will not deviate from a constant too much in small enough domain. Thus one can perform localization argument to recover Stokes estimates.

The proof of Theorem \ref{thm2.4} is organized as follows. First, we restrict ourselves to the case of null initial data, {\it i.e.}, $u_0\equiv 0$, and prove the {\it a priori} estimates for $(u, P)$ under the assumption that $\theta$ is independent of time. Next, we prove the similar estimates for time-dependent temperature. Finally, we derive the desired estimates  for the general initial data $u_0\in D^{1-\frac 1p, p}_{A_r}$.

\subsubsection{Existence of solution for null initial data} We divide the proof into the following three steps.

\noindent
{\textit{(a) A priori estimates with time-independent temperature}}
\begin{theorem}\label{thm2.5}
Suppose $p$, $q$, $r$, $\mu$ and $\Omega$ satisfy the assumptions in Theorem \ref{thm2.4}, $u_0=0$ and $\theta=\theta(x)\in W^{1,q}$. If $(u, P)$ is a smooth solution to system \eqref{2.2} on $\Omega\times[0,T)$, then for any $t<T$ it holds that
\begin{align}\label{n2.6}
\begin{split}
&\big\|(\partial_t u, \Delta u, \nabla P)\big\|_{L^p_t(L^r)}\\
&\quad\le C\Big(B_\theta^{1+\tilde\zeta}\|f\|_{L^p_t(L^r)}+\big(B_\theta-1\big)^{2r'(\tilde\zeta\cdot\frac{q}{q-N}+1)}\|u\|_{L^p_t(L^r)}\Big),
\end{split}
\end{align}
where $B_\theta=1+\|\nabla \theta\|_{L^q(\Omega)}^{\frac q{q-N}}$, $\tilde\zeta=\max\{0, \frac Np-\frac Nr\}$ and $C$ is independent of $f$, $\theta$ and $t$.
\end{theorem}
\begin{proof}
Rewriting \eqref{2.2} as \eqref{n2.5} and applying Theorem \ref{thm2.3}, we obtain
\begin{align}\label{n2.7}
\begin{split}
&\|u\|_{L^\infty_T\big(D^{1-\frac 1p, p}_{A_r}\big)}+\|(\partial_t u,\Delta u,\nabla P)\|_{L^p_T(L^r)}  \\
&\quad\le C\big(\|f\|_{L^p_T(L^r)}+\|\mu(\theta)-\mu(\underline\theta)\|_{L^\infty}\|\Delta u\|_{L^p_T(L^r)}+\|\nabla u\nabla\theta\|_{L^p_T(L^r)}\big),
\end{split}
\end{align}
where $\underline\theta=\inf_{x\in\Omega}\theta(x)$ and $C$ depends on $p$, $q$, $r$, $\Omega$, $\bar\mu$,  $\underline\mu$ and $\bar\mu'$. From now on, we will keep this dependence of $C$ in silence unless otherwise claimed, the value of $C$ may change from line to line.

By Gagliardo-Nirenberg interpolation inequality, Poincar\'e-Wirtinger inequality and Young's inequality,  we arrive at
\begin{equation}\label{2.8}
\|\nabla u\nabla\theta\|_{L^p_T(L^r(\Omega))}\le \epsilon \|\Delta u\|_{L^p_T(L^r(\Omega))}+\epsilon^{\frac {N+q}{N-q}}\|\nabla\theta\|_{L^q(\Omega)}^{\frac {2q}{q-N}}\|u\|_{L^p_T(L^r(\Omega))},
\end{equation}
for any $\epsilon>0$.

On the other hand, since $q>N$, $\theta\in W^{1,q}(\Omega)\hookrightarrow \dot{C}^{\alpha}(\Omega)$ for $\alpha=1-\frac N q\in(0,1)$, we have 
\begin{equation}\label{n2.9}
\|\mu(\theta)-\mu(\underline\theta)\|_{L^\infty}\le \bar{\mu'}\|\theta-\underline\theta\|_{L^\infty(\Omega)}
\le \bar{\mu'}\delta^\alpha(\Omega)\|\theta\|_{\dot{C}^\alpha(\Omega)}\le \bar{\mu'}\delta^\alpha(\Omega)\|\nabla\theta\|_{L^{q}(\Omega)}.
\end{equation}
If $C\delta^\alpha(\Omega)\|\nabla\theta\|_{L^{q}(\Omega)}\le\frac 14$, then the corresponding term is absorbed by the LHS of \eqref{n2.7}. Choosing $\epsilon=\frac 1{4C}$ and substituting \eqref{2.8} and \eqref{n2.9} into \eqref{n2.7}, we obtain
\begin{align}\label{n2.10}
\begin{split}
&\|u\|_{L^\infty_T\big(D^{\frac 1{p'},p}_{A_r}\big)}+\|(\partial_t u,\Delta u, \nabla P)\|_{L^p_T(L^r(\Omega))}\\
&\quad\le C\big(\|f\|_{L^p_T(L^r(\Omega))}+\|\nabla\theta\|_{L^q(\Omega)}^{\frac {2q}{q-N}}\|u\|_{L^p_T(L^r(\Omega))}\big).
\end{split}
\end{align}
Otherwise if $C\delta^\alpha(\Omega)\|\nabla\theta\|_{L^{q}(\Omega)}>\frac 14$,  we perform the space localization to adjust $\delta(\Omega)$. 

We consider the following subordinate partition of $\Omega$:

 $\{\Omega_k\}_{k=1}^K$ is an open covering of $\Omega$ with multiplicity $m$,\footnote{The multiplicity of an covering means at most how many subsets intersect with each other, this quantity only depends of the space dimension $N$.} $\Omega_k$ is star-shaped and for $1\le k\le K$, it holds that 
 $$\delta(\Omega_k)\le \lambda\in(0,\delta(\Omega)),$$
 the value of $\lambda $ will be determined later. 
$\{\phi_k\}_{k=1}^K$ is a family of characteristic function such that $$0\le\phi_k\le 1, ~~\phi_k\in C^2_c(\Omega_k),\quad\sum_{k=1}^K\phi_k(x)=1, ~~\forall~ x\in\bar\Omega, $$ $$\|\nabla ^\alpha \phi_k\|_{L^\infty(\Omega_k)}\le C\lambda^{-|\alpha|},~~ |\alpha|\le 2.$$
The number $K$ of the covering is of order $ \big(\delta(\Omega)\lambda^{-1}\big)^N$,\\
and the number $K'$ of domains $\Omega_k$ intersecting with $\partial\Omega$ is of order $ \big(\delta(\Omega)\lambda^{-1}\big)^{N-1}$.

Now define $u_k=u\phi_k$, $P_k=P\phi_k$ and $f_k=f\phi_k$. Then $(u_k, P_k, f_k)$ satisfies the following system
\begin{align}\label{2.9}
\begin{cases}
& \partial_t u_k-\mu(\underline\theta_k)\Delta u_k+\nabla P_k=f_k-2\mu(\theta)\nabla u\cdot\nabla \phi_k+\mu'(\theta)\mathcal{D}(u)\cdot\nabla\theta\phi_k\\
&\qquad\qquad-\mu(\theta)u\Delta\phi_k+P\nabla\phi_k+(\mu(\theta)-\mu(\underline\theta_k))\Delta u_k,\\
&\divv u_k=u\cdot\nabla \phi_k,\quad \int_\Omega P_k\,dx=0,\\
& u_k|_{t=0}=0,\quad u_k|_{\partial\Omega}=0,
\end{cases}
\end{align}
where $\underline\theta_k=\inf_{x\in\Omega_k}\theta(x)$.
Notice that $u_k$ is not divergence-free and the localization procedure produces some additional lower order terms. To obtain the estimates of the above system, we use a theorem proved by Danchin:
\begin{theorem}[Theorem 3.6 in \cite{Danchin2006}]\label{thm2.6}
Let $\Omega$ be a $C^{2+\epsilon}$ bounded domain of $\mathbb{R}^N$ and $1<p, r<\infty$. Let $\Omega'\subset\bar\Omega$ be open and star-shaped with respect to small ball of diameter $d>0$. Let $\tau\in L^p(0,T;W^{1,r})$ satisfy $\tau(0,\cdot)\equiv 0$,
\begin{equation*}\label{2.10}
\int_\Omega\tau\,dx=0,\quad \partial_t\tau=\tau_0+\divv R,\quad\text{and}\quad\forall~ t\in(0,T),\quad\text{supp} ~\tau_0(t,\cdot)~\cap ~supp~ R(t,\cdot)\subset\bar{\Omega}',
\end{equation*}
with $R$ and $\tau_0$ in $L^p(0,T;L^r(\Omega))$ and $R\cdot n$ in $L^p(0,T;L^r(\partial\Omega))$. Let $v_0\in D^{1-\frac 1p,p}_{A_r}$, $f\in L^p(0,T;L^r(\Omega))$ and $\mu$ is a constant. Then the following system 
\begin{align*}
\begin{cases}
\partial_t v-\mu\Delta v+\nabla P=f,\\
\divv v=\tau,\quad \int_\Omega P\,dx=0,\\
v|_{t=0}=v_0,\quad v|_{\partial\Omega}=0,
\end{cases}
\end{align*}
has a unique solution $(v, P)$ on $\Omega\times[0,T)$ such that 
\begin{equation*}
v\in L^p(0,T;W^{2,r})\cap W^{1,p}(0,T;L^r)\quad\text{and}\quad P\in L^p(0,T;W^{1,r}).
\end{equation*}
Besides, the following estimate holds true with $C=C(r,p,N,\sigma(\Omega))$:
\begin{align*}
 \|(\partial_t v,\mu\nabla^2 v,\nabla P)\|_{L^p_T(L^r(\Omega))}\le&  C\big(\mu^{1-\frac 1p}\|v_0\|_{D^{1-\frac 1p,p}_{A_r}}+\|f\|_{L^p_T(L^r(\Omega))}+\|R\|_{L^p_T(L^r(\Omega))}\\&
+\mu\|\nabla \tau\|_{L^p_T(L^r(\Omega))}+\delta(\Omega')\|\tau_0\|_{L^p_T(L^r(\Omega))}\\
&+\delta^{\frac 1r}(\Omega')\|R\cdot n\|_{L^p_T(L^r(\partial\Omega))}\big).
\end{align*}
\end{theorem}
Let $\tau=u\cdot\nabla\phi_k$, then $\tau(0,\cdot)=u(0,\cdot)\nabla \phi_k\equiv 0$, $\int_\Omega \tau\,dx=\int_\Omega \divv u_k\,dx=0$. Moreover,
\begin{align*}
\partial_t\tau=&\partial_t u\cdot\nabla\phi_k\\
=&\mu(\theta)\Delta u\cdot\nabla\phi_k+\mu'(\theta)\mathcal{D}(u)\nabla\theta\nabla\phi_k-\nabla P\cdot\nabla\phi_k+f\nabla\phi_k\\
=& \underbrace{f\nabla\phi_k+P\Delta\phi_k-\mu(\theta)\mathcal{D}(u)\Delta\phi_k}_{\tau_0}+\divv(\underbrace{\mu(\theta)\mathcal{D}(u)\cdot\nabla\phi_k-P\nabla\phi_k}_{R}),
\end{align*}
and $supp~\tau_0(t,\cdot)\cap supp~R(t,\cdot)\subset\bar{\Omega}_k$. Hence by Theorem \ref{thm2.6}, there exists a unique solution $(u_k, P_k)$ to \eqref{2.9} satisfying
\begin{align}\label{2.13}
\begin{split}
&\|(\partial_t u_k,\Delta u_k,\nabla P_k)\|_{L^p(0,T;L^r(\Omega))}\\
&\lesssim  \|g_k\|_{L^p_T(L^r(\Omega))}+\|R\|_{L^p_T(L^r(\Omega))}+\lambda\|\tau_0\|_{L^p_T(L^r(\Omega))}+\lambda^{\frac 1r}\|R\cdot n\|_{L^p_T(L^r(\partial\Omega))}\\
&\quad+\|\nabla \tau\|_{L^p_T(L^r(\Omega))}+\|\mu(\theta)-\mu(\underline\theta_k)\|_{L^\infty(\Omega_k)}\|\Delta u_k\|_{L^p_T(L^r(\Omega))},
\end{split}
\end{align}
where $g_k=f_k+\mu'(\theta)\mathcal{D}(u)\nabla\theta\phi_k-2\mu(\theta)\nabla u\cdot\nabla \phi_k-\mu(\theta)u\Delta\phi_k+P\nabla\phi_k$.

First of all,
\begin{equation}\label{2.14}
\|\mu(\theta)-\mu(\underline\theta_k)\|_{L^\infty(\Omega_k)}\|\Delta u_k\|_{L^p_T(L^r(\Omega))}\le\bar{\mu'}\lambda^{\alpha}\|\nabla\theta\|_{L^{q}(\Omega)}\|\Delta u_k\|_{L^p_T(L^r(\Omega))}.
\end{equation}
Choose $\lambda\le \kappa\|\theta\|_{W^{1,q}(\Omega)}^{-\frac 1\alpha}$ with $\kappa<<1$ such that the corresponding term can be absorbed by the LHS of \eqref{2.13}. Next, we move on to evaluate the terms on the RHS of \eqref{2.13} one by one.
\begin{align}\label{2.15}
\begin{split}
\|g_k\|_{L^p_T(L^r(\Omega))}\lesssim &\|f_k\|_{L^p_T(L^r(\Omega))}+\|\nabla u\cdot\nabla\phi_k\|_{L^p_T(L^r(\Omega))}\\
&+\|\nabla u\nabla\theta\phi_k\|_{L^p_T(L^r(\Omega))}+\|u\Delta\phi_k\|_{L^p_T(L^r(\Omega))}\\
&+\|P\nabla\phi_k\|_{L^p_T(L^r(\Omega))}\\
\lesssim & \|f\|_{L^p_T(L^r(\Omega_k))}+\lambda^{-1}\|\nabla u\|_{L^p_T(L^r(\Omega_k))}\\&+\|\nabla u\nabla\theta\|_{L^p_T(L^r(\Omega_k))}+\lambda^{-2}\|u\|_{L^p_T(L^r(\Omega_k))}\\&+\lambda^{-1}\|P\|_{L^p_T(L^r(\Omega_k))}.
\end{split}
\end{align}
Similarly, one also has
\begin{align}
&\|R\|_{L^p_T(L^r(\Omega))}\lesssim \lambda^{-1}\|\nabla u\|_{L^p_T(L^r(\Omega_k))}+\lambda^{-1}\|P\|_{L^p_T(L^r(\Omega_k))},\label{2.16}\\
&\lambda\|\tau_0\|_{L^p_T(L^r(\Omega))}\lesssim \|f\|_{L^p_T(L^r(\Omega_k))}+\lambda^{-1}\|P\|_{L^p_T(L^r(\Omega_k))}+\lambda^{-1}\|\nabla u\|_{L^p_T(L^r(\Omega_k))},\label{2.17}\\
&\lambda^{\frac 1r}\|R\cdot n\|_{L^p_T(L^r(\partial\Omega))}\lesssim \lambda^{-1+\frac 1r}\big(\|\nabla u\|_{L^p_T(L^r(\partial\Omega\cap\Omega_k))}+\|P\|_{L^p_T(L^r(\partial\Omega\cap\Omega_k))}\big).\label{2.18}
\end{align}
Substituting \eqref{2.14}-\eqref{2.18} into \eqref{2.13}, one reaches
\begin{align}\label{2.19}
\begin{split}
&\|(\partial_t u_k, \Delta u_k, \nabla P_k)\|_{L^p_T(L^r(\Omega))}\\
&\quad\lesssim  \|f\|_{L^p_T(L^r(\Omega_k))}+\lambda^{-1}\|\nabla u\|_{L^p_T(L^r(\Omega_k))}+\lambda^{-2}\|u\|_{L^p_T(L^r(\Omega_k))}\\
&\quad\quad+\lambda^{-1}\|P\|_{L^p_T(L^r(\Omega_k))}+\|\nabla u\nabla\theta\|_{L^p_T(L^r(\Omega_k))}\\
&\quad\quad+\lambda^{-\frac 1{r'}}\big(\|\nabla u\|_{L^p_T(L^r(\partial\Omega\cap\Omega_k))}+\|P\|_{L^p_T(L^r(\partial\Omega\cap\Omega_k))}\big).
\end{split}
\end{align}

Next, we sum up the local estimates to obtain the whole estimates. Noticing that
\begin{equation*}
\forall~ z\in L^p_T(L^r(\Omega)),\quad \sum_{k=1}^K\|z\|_{L^p_T(L^r(\Omega_k))}\le m^{\frac 1r}K^\zeta\|z\|_{L^p_T(L^r(\Omega))},
\end{equation*}
where $\zeta=\max\{0,\frac 1p-\frac 1r\}$, summing up \eqref{2.19} over $k$, we have 
\begin{align}\label{2.20}
\begin{split}
&\|(\partial_t u,\Delta u,\nabla P)\|_{L^p_T(L^r(\Omega))}\\
&\quad\lesssim \lambda^{-N\zeta}\big(\|f\|_{L^p_T(L^r(\Omega))}+\lambda^{-1}\|\nabla u\|_{L^p_T(L^r(\Omega))}\\
&\qquad+\lambda^{-2}\|u\|_{L^p_T(L^r(\Omega))}+\lambda^{-1}\|P\|_{L^p_T(L^r(\Omega))}+\|\nabla u\nabla \theta\|_{L^p_T(L^r(\Omega))}\big)\\
&\quad\quad+\lambda^{-(N-1)\zeta}\lambda^{-\frac 1{r'}}\big(\|\nabla u\|_{L^p_T(L^r(\partial\Omega))}+\|P\|_{L^p_T(L^r(\partial\Omega))}\big).
\end{split}
\end{align}
Standard interpolation inequalities enable us to further simplify the RHS of \eqref{2.20}. \\
By Gagliardo-Nirenberg and Young's inequality, it follows that
\begin{equation}\label{2.21}
\|\nabla u\|_{L^r(\Omega)}\le C\big(\eta_1^{-1}\|u\|_{L^r(\Omega)}+\eta_1\|\nabla^2 u\|_{L^r(\Omega)}\big),\quad\forall ~\eta_1>0.
\end{equation}
And according to the trace theorem (page 63 in \cite{Galdi}), one deduces that
\begin{align}
\|P\|_{L^r(\partial\Omega)}\le C\big(\eta_2^{-\frac 1r}\|P\|_{L^r(\Omega)}+\eta_2^{\frac 1{r'}}\|\nabla P\|_{L^r(\Omega)}\big),\quad \forall~\eta_2>0,\label{2.22}\\
\|\nabla u\|_{L^r(\partial\Omega)}\le C\big(\eta_3^{-\frac 1r-1}\|u\|_{L^r(\Omega)}+\eta_3^{\frac 1{r'}}\|\nabla^2 u\|_{L^r(\Omega)}\big),\quad\forall~\eta_3>0\label{2.23}.
\end{align}
Again by \eqref{2.8}, 
\begin{equation}\label{2.24}
\|\nabla u\nabla \theta\|_{L^r(\Omega)}\le C\big(\eta_4\|\nabla^2u\|_{L^r(\Omega)}+\eta_4^{\frac {N+q}{N-q}}\|\nabla\theta\|_{L^q(\Omega)}^{\frac{2q}{q-N}}\|u\|_{L^r(\Omega)}\big),\quad\forall~\eta_4>0.
\end{equation}

Now choose $\eta_1=\epsilon\lambda^{N\zeta+1}$, $\eta_2=\eta_3=\epsilon\lambda^{(N-1)\zeta r'+1}$ and $\eta_4=\epsilon\lambda^{N\zeta}$, with $\epsilon<<1$, then the terms $\|\nabla P\|_{L^p_T(L^r(\Omega))}$ and $\|\nabla^2 u\|_{L^p_T(L^r(\Omega))}$ can be absorbed by the LHS of \eqref{2.20}. Consequently, substituting \eqref{2.21}-\eqref{2.24} into \eqref{2.20}, we reach
\begin{align}\label{2.25}
\begin{split}
&\|(\partial_t u, \Delta u, \nabla P)\|_{L^p_T(L^r(\Omega))}\\
&\lesssim \lambda^{-N\zeta}\|f\|_{L^p_T(L^r(\Omega))}+\lambda^{-2-N\zeta\cdot\frac{2q}{q-N}}\|u\|_{L^p_T(L^r(\Omega))}+\lambda^{-1-N\zeta}\|P\|_{L^p_T(L^r(\Omega))},
\end{split}
\end{align}
where  we have used $r>N$.

It remains to show the pressure estimates in terms of $f$ and $u$ to complete the proof. To this end, we evaluate $P$ by a duality argument
$$
\|P\|_{L^r(\Omega)}=\sup_{
\mbox{\tiny
$\begin{array}{c}\|h\|_{L^{r'}(\Omega)}\le 1, \\
\int_\Omega h\,dx=0\end{array}$
}
}\int_\Omega P h\, dx.
$$
Let 
\begin{equation}\label{2.26}
\Delta v=h,\qquad \partial_\nu v|_{\partial\Omega}=0.
\end{equation}
Then according to Proposition C.1 in \cite{Danchin2006}, we have
\begin{equation}\label{2.27}
\|\nabla v\|_{L^{r'}(\Omega)}\le C(\Omega)\|h\|_{L^{r'}(\Omega)},\quad \|\nabla^2 v\|_{L^{r'}(\Omega)}\le C(\Omega)\| h\|_{L^{r'}(\Omega)}.
\end{equation}
Hence, 
\begin{align}\label{2.28}
\begin{split}
\int_\Omega P h\,dx=&\int_\Omega P\Delta v\,dx =-\int_\Omega \nabla P\cdot\nabla v\,dx\\
=&-\int_\Omega\Big(f-\partial_t u+\divv\big(\mu(\theta)\nabla u\big)\Big)\cdot\nabla v\,dx\\
\lesssim & \|f\|_{L^r(\Omega)}\|\nabla v\|_{L^{r'}(\Omega)}+\|\nabla u\|_{L^r(\partial\Omega)}\|\nabla v\|_{L^{r'}(\partial\Omega)}\\
&+\|\nabla u\|_{L^r(\Omega)}\|\Delta v\|_{L^{r'}(\Omega)}\\
\lesssim & (\|f\|_{L^r(\Omega)}+\|\nabla u\|_{L^r(\Omega)})\|h\|_{L^{r'}(\Omega)}+\|\nabla u\|_{L^r(\partial\Omega)}\|\nabla v\|_{L^{r'}(\partial\Omega)},
\end{split}
\end{align}
where we have used \eqref{2.27}. 

For $\nabla u$, by interpolation and Young's inequality, it holds that
\begin{equation}\label{2.30}
\|\nabla u\|_{L^r(\Omega)}\lesssim \epsilon_1^{-1}\|u\|_{L^r(\Omega)}+\epsilon_1\|\nabla^2 u\|_{L^r(\Omega)},\quad\forall~\epsilon_1>0.
\end{equation}

Finally, one can use trace theorem to simplify the boundary terms as 
\begin{align}
& \|\nabla v\|_{L^{r'}(\partial\Omega)}\lesssim \|\nabla v\|_{L^{r'}(\Omega)}^{\frac 1r}\big(\|\nabla^2 v\|_{L^{r'}(\Omega)}+\|\nabla v\|_{L^{r'}(\Omega)}\big)^{\frac 1{r'}}\lesssim\|h\|_{L^{r'}(\Omega)},\label{2.31}\\
& \|\nabla u\|_{L^r(\partial\Omega)}\lesssim \epsilon_2^{-1-\frac 1r}\|u\|_{L^r(\Omega)}+\epsilon_2^{\frac 1{r'}}\|\nabla^2 u\|_{L^r(\Omega)},\quad\forall~\epsilon_2>0.\label{2.32}
\end{align}
Substituting \eqref{2.30}-\eqref{2.32} into \eqref{2.28}, one obtains
\begin{align}\label{2.33}
\|P\|_{L^r(\Omega)}\lesssim & \|f\|_{L^r(\Omega)}+(\epsilon_1^{-1}+\epsilon_2^{-1-\frac 1r})\|u\|_{L^r(\Omega)}+(\epsilon_1+\epsilon_2^{\frac 1{r'}})\|\nabla^2 u\|_{L^r(\Omega)}.
\end{align}

Plugging \eqref{2.33} into \eqref{2.25},  and choosing 
$$\epsilon_1=\kappa\lambda^{1+N\zeta},\quad\epsilon_2=\kappa \lambda^{r'+r'N\zeta},$$ 
where $\kappa<<1$, then $\|\nabla^2 u\|_{L^p_T(L^r(\Omega))}$ can be absorbed by the LHS of \eqref{2.25}, which implies that
\begin{equation}\label{n2.32}
\begin{split}
\|(\partial_t u, \Delta u,\nabla P)\|_{L^p_T(L^r(\Omega))}\lesssim &\lambda^{-N\zeta-1}\|f\|_{L^p_T(L^r(\Omega))}\\
&+\lambda^{-2r'(1+N\zeta \frac {q}{q-N})}\|u\|_{L^p_T(L^r(\Omega))}.
\end{split}
\end{equation}

Combining \eqref{n2.10} with \eqref{n2.32}, one finally obtains \eqref{n2.6}.
The proof of Theorem \ref{thm2.5} is completed.
\end{proof}

\noindent
{\it{(b) A priori estimates with time-dependent temperature}}

Based on Theorem \ref{thm2.5}, we generalize the above results to the case of time-dependent temperature. The main result is the following.

\begin{theorem}\label{thm2.7}
Suppose $p$, $q$, $r$, $\Omega$ and $\mu$ satisfy the assumptions in Theorem \ref{thm2.4}, $u_0=0$ and the temperature satisfies
\begin{equation}
\theta\in \dot{C}^\beta([0,T]; L^\infty(\Omega))\cap L^\infty(0,T;W^{1,q}(\Omega)),
\end{equation} 
for some $\beta\in(0,1)$. If $(u, P)$ is a smooth solution to \eqref{2.2} on $\Omega\times [0,T)$, then it holds that
\begin{equation}\label{2.39}
\|(\partial_t u,\Delta u, \nabla P)\|_{L^p_T(L^r(\Omega))}\le C\big(B^{1+\tilde\zeta}_\theta (T)\|f\|_{L^p_t(L^r(\Omega))}+C_\theta(T)\|u\|_{L^p_t(L^r(\Omega))}\big),
\end{equation}
where $B_\theta(t)$, $C_\theta(t)$ and $C$ are defined as in Theorem \ref{thm2.4}.
\end{theorem}

\begin{proof}
First, rewrite \eqref{2.2} as the following system
\begin{align}\label{2.40}
\begin{cases}
&\partial_t u-\divv\big(\mu(\theta_0)\mathcal{D}(u)\big)+\nabla P=f+(\mu(\theta)-\mu(\theta_0))\Delta u\\
&\quad\quad\quad\quad+\big(\mu'(\theta)\nabla\theta-\mu(\theta_0)\nabla\theta_0\big)\mathcal{D}(u),\\
&\divv u=0,\quad \int_\Omega P\,dx=0,\\
& u|_{t=0}=0,\quad u|_{\partial\Omega}=0.
\end{cases}
\end{align}
Applying Theorem \ref{thm2.5} to system \eqref{2.40}, we have for any $t<T$
\begin{align}\label{2.41}
\begin{split}
&\|(\partial_t u,\Delta u,\nabla P)\|_{L^p_t(L^r(\Omega))}\\
&\quad\lesssim  B_{\theta_0}^{1+\tilde\zeta}\big(\|f\|_{L^p_t(L^r(\Omega))}+\|\big(\mu(\theta)-\mu(\theta_0)\big)\Delta u\|_{L^p_t(L^r(\Omega))}\\
&\qquad+\|(\mu'(\theta)\nabla\theta-\mu'(\theta_0)\nabla\theta_0)\nabla u\|_{L^p_t(L^r(\Omega))}\big)\\
&\qquad+(B_{\theta_0}-1)^{2r'(1+\tilde\zeta\cdot\frac q{q-N})}\|u\|_{L^p_t(L^r(\Omega))}.
\end{split}
\end{align}
Notice that 
\begin{align*}
&\|(\mu'(\theta)\nabla\theta-\mu'(\theta_0)\nabla\theta_0)\nabla u\|_{L^p_t(L^r(\Omega))}\\
&\quad\lesssim  \epsilon \|\nabla^2 u\|_{L^p_t(L^r(\Omega))}+\epsilon^{\frac {N+q}{N-q}}(\|\nabla\theta\|_{L^\infty_t(L^q(\Omega))}+\|\nabla\theta_0\|_{L^q(\Omega)})^{\frac{2q}{q-N}}\|u\|_{L^p_t(L^r(\Omega))}.
\end{align*}
Thus choose $\epsilon=\kappa B_{\theta_0}^{-1-\tilde\zeta}$ with $\kappa<<1$ such that $\nabla^2 u$ can be absorbed by the LHS of \eqref{2.41}. On the other hand, 
\begin{align*}
\|\big(\mu(\theta)-\mu(\theta_0)\big)\Delta u\|_{L^p_t(L^r(\Omega))}\le& \bar{\mu'}\|\theta-\theta_0\|_{L^\infty(\Omega_t)}\|\Delta u\|_{L^p_t(L^r(\Omega))}\nonumber\\
\le&\bar{\mu'}t^\beta\|\theta\|_{\dot C^{\beta}_t(L^\infty(\Omega))}\|\Delta u\|_{L^p_t(L^r(\Omega))}.
\end{align*}

Substituting the above two inequalities into \eqref{2.41}, we obtain
\begin{align}\label{n2.38}
\begin{split}
&\|(\partial_t u,\Delta u,\nabla P)\|_{L^p_t(L^r(\Omega))}\\
&\quad\lesssim  B_{\theta_0}^{1+\tilde\zeta}\big(\|f\|_{L^p_t(L^r(\Omega))}+t^\beta\|\theta\|_{\dot{C}^\beta_t(L^\infty)}\|\Delta u\|_{L^p_t(L^r(\Omega))}\big)\\
&\quad\quad +B_\theta^{(1+\tilde\zeta)\frac{2q}{q-N}}(t)(B_{\theta}(t)-1)^{(1+\tilde\zeta)\cdot\frac {2q r'}{q-N}+2}\|u\|_{L^p_t(L^r(\Omega))},
\end{split}
\end{align}
where $B_\theta(t)=1+\|\nabla\theta\|_{L^\infty_t(L^q(\Omega))}^{\frac q{q-N}}$.

If $t^\beta\|\theta\|_{\dot{C}^{\beta}_t(L^\infty(\Omega))}B^{1+\tilde\zeta}_{\theta_0}\le \frac 1{2C}$, then the second term on the RHS of \eqref{n2.38} can be absorbed by the LHS, which gives the desired estimates.

Otherwise, if $t^\beta\|\theta\|_{\dot{C}^{\beta}_t(L^\infty(\Omega))}B^{1+\tilde\zeta}_{\theta_0}> \frac 1{2C}$, then we perform time localization to adjust the time interval. Specifically, choose 
$$\tau=\min\{T,\kappa\big(B_{\theta}(T)^{1+\tilde\zeta}\|\theta\|_{\dot{C}^{\beta}_T(L^\infty(\Omega))}\big)^{-1/\beta}\},\quad \kappa<<1.$$
Then for any $t\in[0,\tau]$, it holds that
\begin{equation}\label{2.42}
\|(\partial_t u,\Delta u, \nabla P)\|_{L^p_t(L^r(\Omega))}\lesssim B^{1+\tilde\zeta}_{\theta_0}\|f\|_{L^p_t(L^r(\Omega))}+B_\theta^{(1+\tilde\zeta)\frac{2q}{q-N}}(t)(B_\theta(t)-1)^{r^*}\|u\|_{L^p_t(L^r(\Omega))},
\end{equation}
where $r^*=(1+\tilde\zeta)\cdot\frac {2q r'}{q-N}+2$.

Next, we try to extend the above estimates to $[0,T]$. To this end, we perform a partition on time interval as the following:
 
 Suppose $\{\psi_k\}_{k\in\mathbb{N}}$ is a partition of unity of $\mathbb{R}_+$ such that
 $$supp~\psi_0\subset [0,\tau],\quad
  \psi_0\equiv 1~\text{on}~[0,\frac \tau 2],$$
$$supp~\psi_k \subset (\frac k 2\tau,\frac k 2\tau+\tau)~for~k\ge 1,\quad \|\partial_t\psi_k\|_{L^\infty}\le c\tau^{-1},$$
$$\sum_{k=0}^K\psi_k(t)=1,~\forall ~t\in[0,T],\quad\frac K 2\tau\le T<\frac {K+1}2\tau.$$

Denote $u_k=u\psi_k$, $P_k=P\psi_k$, $f_k=f\psi_k$, then $(u_k, P_k, f_k)$ satisfy the following system
\begin{align}\label{2.43}
\begin{cases}
\partial_t u_k-\divv\big(\mu(\theta)\mathcal{D}(u_k)\big)+\nabla P_k=f_k+u\partial_t\psi_k,\\
\divv u_k=0,\quad \int_\Omega P_k\,dx=0,\\
u_k|_{t=\frac k 2\tau}=0, \quad u_k|_{\partial\Omega}=0.
\end{cases}
\end{align}

Let $I_k=[\frac k2\tau,\frac k2 \tau+\tau]$, $k=0,\cdots, K-1$, $I_K=[\frac K2\tau, T]$.

For any $t\in I_k$, it follows from \eqref{2.42} that
\begin{align}\label{2.44}
\begin{split}
&\|(\partial_t u_k,\Delta u_k,\nabla P_k)\|_{L^p(\frac k2\tau, t; L^r(\Omega))}\\
&\quad\lesssim B_{\theta}^{1+\tilde\zeta}(t)\big(\|f_k\|_{L^p(\frac k2\tau, t; L^r(\Omega))}+\|u\partial_t\psi_k\|_{L^p(\frac k2\tau, t; L^r(\Omega))}\big)\\
&\quad\quad+B_\theta^{(1+\tilde\zeta)\frac{2q}{q-N}}(t)(B_\theta(t)-1)^{r^*}\|u_k\|_{L^p(\frac k2\tau, t; L^r(\Omega))}.
\end{split}
\end{align}

Notice that 
\begin{align}\label{2.45}
\begin{split}
\|u\partial_t\psi_k\|_{L^p(\frac k2\tau, t; L^r(\Omega))}\lesssim &\tau^{-1}\|u\|_{L^p(\frac k2\tau, t; L^r(\Omega))}\\
\lesssim &\big(B_{\theta}(T)^{1+\tilde\zeta}\|\theta\|_{\dot{C}^{\beta}_T(L^\infty)}\big)^{1/\beta}\|u\|_{L^p(\frac k2\tau, t; L^r(\Omega))}.
\end{split}
\end{align}

Substituting \eqref{2.45} into \eqref{2.44}, then summing up over $k$ gives that
\begin{align}\label{2.46}
\begin{split}
&\|(\partial_t u,\Delta u,\nabla P)\|_{L^p_T(L^r(\Omega))}\lesssim  B_\theta^{1+\tilde\zeta}(T)\|f\|_{L^p_T(L^r(\Omega))}\\
&\quad+\big(B_\theta^{(1+\tilde\zeta)\frac{2q}{q-N}}(T)(B_\theta(T)-1)^{r^*}+B_\theta(T)^{(1+\tilde\zeta)(1+\frac 1\beta)}\|\theta\|_{ \dot{C}^\beta_T(L^\infty)}^{1/\beta}\big)\|u\|_{L^p_T(L^r(\Omega))}.
\end{split}
\end{align}
By the definition of $C_\theta(t)$, \eqref{2.39} is proved. We remark that $l_2\ge\min\{r^*\frac {q}{q-N}, \frac 1\beta\}>1$.

\end{proof}

\noindent
{\it(c) Existence and uniqueness of solution to \eqref{2.2} with null initial data}
 
\begin{theorem}\label{thm2.8}
Suppose all the assumptions in Theorem \ref{thm2.4} are true and $u_0=0$, then the system \eqref{2.2} has a unique strong solution $(u, P)$ on $\Omega_T$ satisfying
\begin{equation}\label{2.47}
\|(\partial_t u,\Delta u,\nabla P)\|_{L^p_t(L^r(\Omega))}\le CB_\theta^{1+\tilde\zeta}(t)\exp(CtC_\theta(t))\|f\|_{L^p_t(L^r(\Omega))},
\end{equation}
for any $t<T$, where $B_\theta(t)$, $C_\theta(t)$ and $C$ are defined as in Theorem \ref{thm2.4}.
\end{theorem}
\begin{proof}
The proof of existence of local strong solution is trivial after {\textit{a priori}} estimates \eqref{2.47}, Thus we only give the proof of this estimates. 

Suppose $(u, P)$ is a smooth solution on $\Omega_T$, indeed, we have  
\begin{equation}\label{2.48}
\dfrac d{dt}\|u\|_{L^r(\Omega)}\le \|\partial_t u\|_{L^r(\Omega)}.
\end{equation}
Taking advantage of \eqref{2.46} and \eqref{2.48}, then for all $\epsilon>0$
\begin{align}\label{2.49}
\begin{split}
\|u(t,\cdot)\|_{L^r(\Omega)}^p=&p\int_0^t\|u(\tau,\cdot)\|^{p-1}_{L^r(\Omega)}\dfrac d{dt}\|u(\tau,\cdot)\|_{L^r(\Omega)}d\tau\\
\le&(p-1)\epsilon\int_0^t\|u(\tau,\cdot)\|_{L^r(\Omega)}^pd\tau+\epsilon^{1-p}\int_0^t\|\partial_t u\|_{L^r(\Omega)}^pd\tau\\
\le&\big((p-1)\epsilon+\epsilon^{1-p}CC_\theta^p(t)\big)\|u\|_{L^p_t(L^r(\Omega))}^p\\
&+\epsilon^{1-p}CB_\theta^{(1+\tilde\zeta)p}(t)\|f\|_{L^p_t(L^r(\Omega))}^p.
\end{split}
\end{align}
Now, let us denote 
\begin{align*}
& \alpha(t)=\epsilon+\epsilon^{1-p}C_\theta^p(t),\quad \beta(t)=\epsilon^{1-p}B_\theta^{(1+\tilde\zeta)p}(t)\|f\|_{L^p_t(L^r(\Omega))}^p,\\
& F(t)=\|u\|_{L^p_t(L^r(\Omega))}^p,\quad \gamma(t)=\int_0^t\beta(s)ds,
\end{align*}
then from \eqref{2.49}, it follows that
\begin{equation}
F(t)\le C\int_0^t\alpha(s)F(s)ds+C\gamma(t).
\end{equation}
Therefore, Gronwall's lemma implies
\begin{align}\label{2.51}
\begin{split}
F(t)\le & C\gamma(t)\exp(C\int_0^t\alpha(s)ds)\\
\le & C\epsilon^{1-p}tB^{(1+\tilde\zeta)p}_\theta(t)\|f\|_{L^p_t(L^r(\Omega))}\exp(C\epsilon t+C\epsilon^{1-p}tC_\theta^p(t)).
\end{split}
\end{align}
 Choosing $\epsilon=C_\theta(t)$, then \eqref{2.51} turns into
\begin{equation}\label{2.52}
\int_0^t\|u(\tau,\cdot)\|_{L^r(\Omega)}^pd\tau\lesssim C_\theta^{-p}(t)B^{(1+\tilde\zeta)p}_\theta(t)\|f\|^p_{L^p_t(L^r(\Omega))}\exp(CtC_\theta(t)),
\end{equation}
where we have used $tC_\theta(t)\lesssim \exp(tC_\theta(t))$.

Inserting \eqref{2.52} into \eqref{2.39}, we finally obtain \eqref{2.47}. 
\end{proof}

\subsubsection{General initial data}

In this section, we generalize the previous result to the case of general initial data. Consider the following two systems:
\begin{align}\label{2.53}
\begin{cases}
\partial_t\omega-\mu(\underline\theta)\Delta \omega+\nabla\Pi=f,\\
\divv \omega=0,\quad \int_\Omega \Pi\, dx=0,\\
\omega|_{t=0}=u_0,\quad \omega|_{\partial\Omega}=0,
\end{cases}
\quad
\begin{cases}
\partial_t v-\divv\big(\mu(\theta)\mathcal{D}(v)\big)+\nabla Q\\
=\big(\mu(\theta)-\mu(\underline\theta)\big)\Delta\omega+\mu'(\theta)\nabla\theta\mathcal{D}(\omega),\\
\divv v=0,\quad \int_{\Omega} Q\,dx=0,\\
v|_{t=0}=0,\quad v|_{\partial\Omega}=0,
\end{cases}
\end{align}
where $\underline\theta=\inf_{x\in\Omega}\theta_0$. It's easy to verify that  $u=v+\omega$ and $P=\Pi+Q$ satisfy system \eqref{2.2} if $\omega$ and $v$ satisfy the corresponding system.

Theorem \ref{thm2.3} implies that there exists a unique solution $(\omega, \Pi)$ to the first system of \eqref{2.53} such that for any $0<t<T$
\begin{equation}\label{2.54}
\|(\partial_t\omega,\Delta\omega,\nabla\Pi)\|_{L^p_t(L^r(\Omega))}+\|\omega\|_{L^\infty_t\big(D^{1-\frac 1{p},p}_{A_r}\big)}\lesssim \|f\|_{L^p_t(L^r(\Omega))}+\|u_0\|_{D^{1-\frac 1{p}, p}_{A_r}}.
\end{equation}
On the other hand, Theorem \ref{thm2.8} implies that there exists a unique solution $(v, Q)$ to the second system of \eqref{2.53} such that
\begin{align}\label{2.55}
\begin{split}
&\|(\partial_t v,\Delta v,\nabla Q)\|_{L^p_t(L^r(\Omega))}+\|v\|_{L^\infty_t(D^{1-\frac 1{p},p}_{A_r})}\\
\lesssim & B^{1+\tilde\zeta}_\theta(t)\exp(CtC_\theta(t))\big(\|(\mu(\theta)-\mu(\underline\theta))\Delta \omega\|_{L^p_t(L^r(\Omega))}+\|\nabla\theta\nabla\omega\|_{L^p_t(L^r(\Omega))}\big)\\
\lesssim & B^{2+\tilde\zeta}_\theta(t)\exp(CtC_\theta(t))(\|f\|_{L^p_t(L^r(\Omega))}+\|u_0\|_{D^{1-\frac 1 {p},p}_{A_r}}),
\end{split}
\end{align}
where we have used \eqref{2.54}.

Adding up \eqref{2.54} and \eqref{2.55} yields that
\begin{align}\label{2.57}
\begin{split}
& \|u\|_{L^\infty_t(D^{1-\frac 1{p}, p}_{A_r})}+\|(\partial_t u,\Delta u,\nabla P)\|_{L^p_t(L^r(\Omega))}\\
&\quad\lesssim  B^{2+\tilde\zeta}_\theta(t)\exp(CtC_\theta(t))(\|u_0\|_{D^{1-\frac 1{p},p}_{A_r}}+\|f\|_{L^p_t(L^r(\Omega))}).
\end{split}
\end{align}
Thus, \eqref{2.3} is proved.

To prove \eqref{2.4}, we apply Theorem \ref{thm2.7} to the second system of \eqref{2.53},
\begin{align}\label{n2.57}
\begin{split}
&\|(\partial_t v,\Delta v, \nabla v)\|_{L^p_t(L^r(\Omega))}\\
&\quad\lesssim B^{1+\tilde\zeta}_\theta(t)\Big(\big\|\big(\mu(\theta)-\mu(\underline\theta)\big)\Delta\omega\big\|_{L^p_t(L^r(\Omega))}+\|\nabla\theta\nabla\omega\|_{L^p_t(L^r(\Omega))}\Big)\\
&\qquad+C_\theta(t)\|v\|_{L^p_t(L^r(\Omega))}\\
&\quad\lesssim B^{2+\tilde\zeta}_\theta(t)(\|f\|_{L^p_t(L^r(\Omega))}+\|u_0\|_{D^{1-\frac 1p,p}_{A_r}})+C_\theta(t)\|v\|_{L^p_t(L^r(\Omega))}.
\end{split}
\end{align}
Summing up \eqref{2.54} and \eqref{n2.57} and noticing that
\begin{equation*}
\|v\|_{L^p_t(L^r(\Omega))}\le \|u\|_{L^p_t(L^r(\Omega))}+\|\omega\|_{L^p_t(L^r(\Omega))},
\end{equation*}
we finally obtain \eqref{2.4}. This completes the proof of Theorem \ref{thm2.4}.
\section*{Acknowledgements}
This work was done  when the second author was visiting the School of Mathematics and Statistics of Beijing Institute of Technology in 2016. He would like to thank the school for its hospitality and financial support.

The first author is partially supported by National Nature Science Foundation of China under the grants 11501028 and 11471323 respectively, Postdoctoral Science Foundation of China under the grant 2016T90038 and Basic Research Foundation of Beijing Institute of Technology under the grant 20151742001

And both authors would like to thank Prof. Zhouping Xin for his constant caring and helpful advice to our research.



\end{document}